\documentclass[12pt]{amsart}
\usepackage{amsmath,amsthm,amsfonts,amssymb,graphicx,tikz,color,fancyhdr}
\usetikzlibrary{matrix,arrows.meta}
\usepackage{tikz-cd}
\usepackage{enumitem}
\usepackage{soul}
\newcommand{\tensor}{{\otimes}}
  \newcommand{\cS}{{\mathcal S}}
  
  \newcommand{\cF}{{\mathcal F}}

\newcommand{\As}{{\mathcal A}}
\renewcommand{\O}{{\mathcal O}}
\let\paragraph\S
\renewcommand{\S}{{\mathbb S}}
\newcommand{\Ac}{{A_\C}}

\newcommand{\C}{{\mathbb C}}
\renewcommand{\H}{{\mathbb H}}
\newcommand{\N}{{\mathbb N}}

\newcommand{\Z}{{\mathbb Z}}

\newcommand{\R}{{\mathbb R}}

\DeclareMathOperator{\Tr}{{Tr}}
\DeclareMathOperator{\Nm}{{N}}

\newcommand{\Lie}{\mathop{Lie}}

\newcommand{\Aut}{\mathop{{\rm Aut}}}

\newcommand{\dom}{D}

\newcommand{\Gc}{G_{\C}}
\newcommand{\GC}{G_{\C}}

\newcommand{\Hc}{\H_{\C}}
\newcommand{\WC}{{W\tensor_\R\C}}
\newcommand{\VC}{{V_\C}}
\newcommand{\AC}{A\tensor_{\R}\C}
\newcommand{\Sm}{S_\Delta^mV}
\newcommand{\X}{\underline{\overline{X}}}
\newcommand{\uZ}{\underline{\Z}}
\newcommand{\cdiv}{\mathop{cdiv}}

\newcommand{\dfe}{=}
\newtheorem{lemma}{Lemma}[section]
\newtheorem{definition}[lemma]{Definition}

\newtheorem{corollary}[lemma]{Corollary}
\newtheorem{SpecAss}[lemma]{Assumptions}
\newtheorem{example}[lemma]{Example}
\newtheorem*{remark}{Remark}

\newtheorem*{warning}{Warning}

\newtheorem{proposition}[lemma]{Proposition}
\newtheorem{theorem}[lemma]{Theorem}
\begin{document} 
\numberwithin{equation}{section}
\def\labelenumi{\rm(\roman{enumi})}
\newif\ifdebug
\debugfalse
\ifdebug
\let\oldlabel\label
\renewcommand{\label}[1]{\oldlabel{#1}\marginpar{\footnotesize\tt\color{blue}#1}}
\newcommand{\notiz}[1]{{\tt To do:}{\color{red}\tt #1}{\message{XX: #1}}}
\fi
\title{%
Invariants and Automorphisms for Slice Regular Functions}
\author {Cinzia Bisi \& J\"org Winkelmann}
\begin{abstract}
Let $A$ be one of the following Clifford algebras :  $\mathbb{R}_2 \cong \mathbb{H}$ or $\mathbb{R}_3$. For the algebra $A$,
the automorphism group $Aut(A)$ and its invariants are well known. In this paper we will describe the invariants of the automorphism group of the 
algebra of slice regular functions over $A$. 
\end{abstract}
\address{%
Cinzia Bisi \\
Department of Mathematics and Computer Sciences\\
Ferrara University\\
Via Machiavelli 30\\
44121 Ferrara \\
Italy
}
\email{bsicnz@unife.it \newline
 ORCID : 0000-0002-4973-1053
}
\address{%
J\"org Winkelmann \\
IB 3/111\\
Lehrstuhl Analysis II \\
Fakult\"at f\"ur Mathematik\\
Ruhr-Universit\"at Bochum\\
44780 Bochum \\
Germany\\
}
\email{joerg.winkelmann@rub.de\newline
  ORCID: 0000-0002-1781-5842
}
\thanks{
The two authors were partially supported by GNSAGA of INdAM.
C. Bisi was also partially supported by PRIN \textit{Variet\'a reali e complesse:
geometria, topologia e analisi armonica}. 
}
\maketitle
\tableofcontents
\section{Introduction}

The theory of slice regular functions was introduced by G. Gentili and D. Struppa in two seminal papers in 2006 \cite{GS1} and in \cite{GS2} : they used the fact that $\forall I \in \mathbb{S}_{\mathbb{H}} = \{ J \in \mathbb{H} \,\, | \,\, J^2=-1 \}$ the real subalgebra $\mathbb{C}_{I}$ generated by 1 and $I$ is isomorphic to $\mathbb{C}$ and they decomposed the algebra $\mathbb{H}$ into a "book-structure" via these complex "slices" :
$$
\mathbb{H} =\cup_{I \in \mathbb{S}_{\mathbb{H}} }   \mathbb{C}_I
$$
On an open set $\Omega \subset \mathbb{H},$ they defined a differentiable function $f \colon \Omega \to \mathbb{H}$ to be (Cullen or) slice regular if, for each $I \in \mathbb{S},$ the restriction of $f$ to $\Omega_I= \Omega \cap \mathbb{C}_I$
is a holomorphic function from $\Omega_I$ to $\mathbb{H},$ both endowed with the complex structure defined by left multiplication with $I.$
This definition covers all functions given by
convergent power series of the form :
$$
\sum_{n \in \mathbb{N}_0} q^n a_n 
$$
with $\{ a_n \}_{n \in \mathbb{N}_0} \subset \mathbb{H}.$ \\
Later on, the approach introduced by R. Ghiloni and A. Perotti in 2011, \cite{GP}, for an alternative $*$-algebra $A$ over $\mathbb{R}$ makes use of the complexified algebra $A \otimes_{\mathbb{R}} \mathbb{C},$ denoted by $A_{\mathbb{C}}$. \\
Let us denote its elements as $a+\iota b$ where $a, b \in A$ and $\iota$ is to be considered as the imaginary unit of $\mathbb{C}$.

For any slice regular function, and for any $I \in \mathbb{S}_{\mathbb{H}},$ the restriction $f \colon \mathbb{C}_I \to \mathbb{H}$ can be lifted through the map $\phi_I \colon \mathbb{H}_{\mathbb{C}} \to \mathbb{H},$
$\phi_I (a+\iota b) := a+I b$ and it turns out that the lift does not depend on $I.$  In other words, there exists a holomorphic function $F \colon \mathbb{C} \cong \mathbb{R}_{\mathbb{C}} \to \mathbb{H}_{\mathbb{C}}$ which makes the following diagram commutative for all $I \in \mathbb{S}_{\mathbb{H}}.$ \\
$$
\begin{tikzcd}
 & \mathbb{C} \arrow[r, "F"] \arrow[d, "\phi_I"']
& \mathbb{H}_{\mathbb{C}} \arrow[d, "\phi_I"]  \\
& \mathbb{H}  \arrow[r, "f"] & \mathbb{H}  
\end{tikzcd}
$$
After the first definitions were given, the theory of slice regular functions knew a big development : see, among the others, the following references \cite{GSSMono}, \cite{BW2}, \cite{BDMW}, \cite{BW1}, \cite{BW3}, \cite{AB2},\cite{AB1}, \cite{BS}, \cite{BG}, \cite{BM}, \cite{BC23}. \\
Non-commutative
associative
division algebras admit many automorphisms, because
$x\mapsto y^{-1}xy$ is an automorphism for every invertible element $y$.

The ``essential'' properties of a number or a function should not
be changed by automorphisms.

In the case of the algebras under consideration here, there is an
{\em antiinvolution} $x\mapsto \bar x$ which commutes with all
automorphisms. As a consequence, $\Nm(x)=x\bar x$ and $\Tr(x)=x+\bar x$ are
{\em invariant} under automorphisms. In fact, for $A=\H$
we have the equivalence:
\[
\Nm(z)=\Nm(w)\text{ and } \Tr(z)=\Tr(w)
\iff \exists \,\,\, \phi\in \Aut(A):\phi(z)=w
\]
(Here $\phi$ is an automorphism of $A$ as an $\R$-algebra.)

This raises the question whether a similar correspondence holds
not only for the elements in the algebra, but also for slice-regular
functions of this algebra.

As it turns out, essentially this is true, but only via the associated
stem functions and up to a condition on the multiplicity with which
  values in the center of $\Ac$ are assumed. 
  To state the latter condition, in \paragraph\ref{ss-cdiv} 
  we introduce the notion of a
``{\em central divisor}'' $\cdiv$.

More precisely, the conjugation $x\mapsto \bar x$
on the algebra $A$ defines a
conjugation on the space of slice-regular functions
which allows the definition of $\Tr$ and $\Nm$ as before.

This conjugation corresponds to a conjugation on the associated
space of stem functions $F:\C\to A\tensor_\R\C$ defined as
\[
  F^c:z\mapsto \overline{\left (F( z)\right)}
  \]
where $\overline{(q\tensor w)}$ is defined as
$\bar q\tensor \left( w \right)$ (with $q\in A, w\in\C$).
Again, 
conjugation induces $\Tr$ and $\Nm$ as
\[
\Tr(F):z\mapsto (F+F^c)(z),\quad \Nm(F):z\mapsto (F(z))(F^c(z))
\]

With these definitions, the correspondence
between regular functions and stem functions is compatible
with conjugation. As a consequence, if $F$ is the stem function for $f$,
then $F^c$ is the stem function for $f^c$. Moreover
$\Nm(F)$, resp. $\Tr(F)$, are the stem functions for $\Nm(f)$, resp. $\Tr(f)$.

As it turns out,
essentially
$\Nm(F)$, $\Tr(F)$ and $\cdiv(F)$ ( or, equivalently,
$\Nm(f)$, $\Tr(f)$ and $\cdiv(f)$) characterize $F$ (equivalently, $f$)
up to replacing $F$ with $z\mapsto \phi(z) \left( F(z)\right)$ for
some holomorphic map $\phi$  from $\C$ to the automorphism group of
the complex algebra $\AC$. \\

  Here $\cdiv$ is an additional invariant
  which we introduce in \paragraph\ref{ss-cdiv}
  for slice regular functions which are not slice preserving.

  In this paper we describe the group of automorphisms of the algebra of slice regular functions with values in $ \mathbb{H}$
and $ \mathbb{R}_3 \cong \mathbb{H} \oplus \mathbb{H}$.%
\footnote{
In a forthcoming paper we will investigate other algebras as well.
}
\\
Our Main Theorem is the following :
\begin{theorem}\label{mainth}
    Let $\H$ denote the algebra of quaternions,
  $\Hc=\H\tensor_\R\C$, $G=Aut(\H)\cong SO(3,\R)$,
  $\GC=Aut(\Hc)\cong SO(3,\C)$.
  Let $\dom\subset\C$ be a symmetric domain
  and let $\Omega_\dom\subset \H$ denote the corresponding axially symmetric
  domain.

  Let $f,h:\Omega_\dom\to\H$ be slice regular functions and
  let $F,H:\dom\to\Hc$
  denote the corresponding stem functions.

    \begin{enumerate}[label=\alph*)]
  \item
  Assume that neither $f$ nor $h$ are  slice preserving.
  
  Then the following are equivalent:
  \begin{enumerate}[label=(\roman*)]
  \item
    $f$ and $h$ have the same invariants $\cdiv$, $\Tr$, $\Nm$.
  \item
    $F$ and $H$ have the same invariants $\cdiv$, $\Tr$, $\Nm$.
  \item
    $\cdiv(F)=\cdiv(H)$ and for every $z\in\dom$ there exists an element
    $\alpha\in Aut(\Hc)=\GC$ such that $F(z)=\alpha(H(z))$.
  \item
    There is a holomorphic map $\phi:\dom\to\GC$
    such that $F(z)=\phi(z)\left(H(z)\right)\ \forall z\in\dom$.
  \item
      There is a holomorphic map $\alpha:\dom\to\Hc^*$
      such that
      \[
      F(z)=\alpha(z)^{-1}\cdot H(z)\cdot\alpha(z)
      \]
  \end{enumerate}
\item
  Assume that $f$ is slice preserving.
      Then the following are equivalent:
  \begin{enumerate}[label=(\roman*)]
  \item
    $f=h$.
  \item
    $F=H$.
  \item
    For every $z\in\dom$ there exists an element
    $\alpha\in Aut(\Hc)=\GC$ such that $F(z)=\alpha(H(z))$.
\item
    There is a holomorphic map $\phi:\dom\to\GC$
    such that $F(z)=\phi(z)\left(H(z)\right)\ \forall z\in\dom$.
  \item
      There is a holomorphic map $\alpha:\dom\to\Hc^*$
      such that
      \[
      F(z)=\alpha(z)^{-1}\cdot H(z)\cdot\alpha(z)
      \]
   \end{enumerate}
    \end{enumerate}
    \end{theorem}

\begin{remark}
  The notion of a ``central divisor'' is defined only if the function
  is not slice preserving. This is similar to the ordinary complex situation
  where the divisor of a holomorphic function is defined only if it is not
  constantly zero.
\end{remark}

Theorem~\ref{mainth} is proved in \paragraph\ref{pf-main}.

We also derive a corresponding result
for the Clifford algebra $\R_3$ (which as
$\R$-algebra is isomorphic to $\H\oplus\H$)
(Theorem~\ref{thm-r3}).

Quaternions can be used to describe orthogonal complex structures (OCS)
on $4$-dimensional Euclidean space, since their imaginary units pa\-ra\-metrize automorphisms of $\mathbb{R}^4 \cong \mathbb{C}^2$. An injective slice regular function  on a symmetric slice domain of $\mathbb{H}$ minus the reals, defines a new OCS via the push-forward of the standard one. It would be very interesting in the authors' opinion to understand if slice regular functions that are in the same orbit, by the action of automorphisms as explained in this paper, induce the same OCS or isomorphic OCS, \cite{GSS}. 

\subsection{Related work}

In \cite{AD1}, Altavilla and de Fabritiis investigated equivalence
relations for {\em semi-regular} functions.

These semi-regular functions are locally $*$-quotient of slice regular
functions and correspond to meromorphic functions in complex analysis.

Semi-regular functions have the advantage that they are always invertible
unless they are zero-divisors. This eases the use of linear algebra.

In \cite{AD1} it is proved that for any two semi-regular functions
$f$ and $g$
the following properties are equivalent:

\begin{itemize}
\item
  $\Tr(f)=\Tr (g)$ and $\Nm(f)=\Nm(g)$.
\item
  There is a semi-regular function $h$
  with $h*f*h^{-*}=g$.
\item
  The ``Sylvester-operator'' $S_{f,-g}:x\mapsto f*x-x*g$
  is not invertible.
\end{itemize}

In comparison, we obtain a stronger conclusion (namely, conjugation
by a {\em regular} function instead by a function which is only
semi-regular), but for this we need an additional assumption, namely
equality not only of trace $\Tr$ and norm $\Nm$, but also of the
``central divisor'' $\cdiv$ which we introduce in \paragraph\ref{ss-cdiv}.
See \paragraph~\ref{an-ex} for an instructive example.

\section{Preparations}\label{sect-prep}

Here we collect basic facts and notions needed for our main result.
First we discuss conjugation, norm and trace,
then types of domains, then slice regular
functions and stem functions, followed by investigating conjugation,
norm and trace for function algebras.

\subsection{Conjugation, norm and trace}

Let $A$ be an
alternative
$\R$-algebra  with $1$
and let $x\mapsto \bar x$ be an {\em
  antiinvolution}, i.e., an $\R$-linear map such that
$\overline{xy}=(\bar y)\cdot(\bar x)$
and $\overline{(\bar x)}=x$ for all $x,y\in A$.
(An $\R$-algebra with an antiinvolution is often called
$*$-algebra.)%
\footnote{In this article we are mainly concerned with
    the algebra of quaternions and $\R_3\cong\H\oplus\H$,
    but we would also like to prepare for
    a future article on other $\R$-algebras.}

\begin{definition}\label{def-n-tr}
  Given an $\R$-algebra $A$ with antiinvolution $x\mapsto\bar x$,
  we define:
  \begin{align*}
    \text{Trace: } & \Tr(x)=x+\bar x\\
    \text{Norm: } & \Nm(x)=x\bar x\\
  \end{align*}
\end{definition}

Consider
\[
C=\{x\in A: x= \bar x\}.
\]

We assume that $C$ is {\em central} and associates with
all other elements, i.e.,
\[
\forall c\in C, x,y\in A:  cx=xc\text{ and }c(xy)=(cx)y
\]
It is easy to verify that $C$ is a subalgebra
(under these assumptions, i.e., if $C$ is assumed to be central).

\begin{lemma}
  Under the above assumptions the following properties hold:
  
  \begin{enumerate}
    \item
      $\forall x\in \R:x=\bar x$.
        \item
    $\forall x\in A: \Nm(x),\Tr(x)\in C=\{ y\in A: y=\bar y\}$
  \item
    $\forall x\in A: x\bar x=\bar x x$
  \item
    $\forall x\in A: \Nm(x)=\Nm(\bar x)$.
  \item
    $\forall x,y\in A: \Nm(xy)=\Nm(x)\Nm(y)$.
  \end{enumerate}
\end{lemma}
\begin{proof}
  \begin{enumerate}
\item
  \[
  \bar 1=1\cdot \bar 1\ \implies\
  1=\overline{1\cdot \bar 1}=1  \cdot\bar 1=\bar 1
  \]
  By $\R$-linearity of the antiinvolution this yields
  $\forall x\in\R:x=\bar x$.

  \item
    \[
    \overline{\left(\Tr(x)\right)}=\overline{x+\bar x}=\bar x+x=\Tr(x)
    \]
    and
    \[
    \overline{\left(\Nm(x)\right)}=
    \overline{\left(x\bar x\right)}=
    \overline{(\bar x)}\bar x=x\bar x=\Nm(x)
    \]
  \item
    $x+\bar x$ is central, hence $x(x+\bar x)=(x+\bar x)x$ which
    implies $x^2+x\bar x=x^2+\bar xx$ and consequently $x\bar x=\bar xx$.
  \item
    $\Nm(\bar x)=(\bar x)\overline{(\bar x)}=(\bar x)x=x(\bar x)=\Nm(x)$.
  \item
    If $A$ is associative we argue as follows:
    \begin{align*}
        \Nm(xy)&=(xy)\overline{(xy)}=xy(\bar y\bar x)\\
        &=x(y\bar y)\bar x=x\bar x(y\bar y)\\
        &=\Nm(x)\Nm(y)
    \end{align*}
    For the general case we observe that $x,\bar x,y,\bar y$
    are all contained in the $C$-algebra $A_0$ generated by $x$ and $y$
    (note that $x+\bar x,y+\bar y\in C$). Artin's theorem
    (\cite{Schafer}, Theorem~3.1) implies that
    $A_0$ is associative. Thus all the calculations in
    the above sequence of equations
    take place {\em within} an associative algebra, namely
    $A_0$ and the proof is therefore still valid, even if $A$ itself is
    not associative, but only alternative.
  \end{enumerate}
\end{proof}

\subsection{Relation with notions in linear algebra
  and number theory}

The norm and trace as considered here for the algebra of
quaternions are closely related to notions
in linear algebra and number theory.

The algebra $\H$ is a {\em central simple $\R$}-algebra
with $\H\tensor_\R\C\cong Mat(2\times 2,\C)$.

In the theory of central simple algebras (see e.g.~\cite{LF},
Chapter 29), one considers {\em reduced traces $Trd$} and {\em reduced
norms $Nrd$} defined as $Tr(\phi(x))$ resp.~$\det(\phi(x))$
for $x\in\H$, where $\phi$ denotes the embedding of $\H$
into $Mat(2\times 2,\C)$ via
the natural embedding $\H\subset\H\tensor_\R\C$ composed with
an isomorphism $\H\tensor_\R\C\cong Mat(2\times 2,\C)$.

There is also a  connection with notions of
  norm and trace in algebraic number theory:

If $A=\C$ and $x\mapsto x^c$ is complex conjugation,
then $\Tr $ and $\Nm$ defined as here agree with the
number-theoretical notions $\Tr $ and $\Nm$
for the Galois field extension $\C/\R$.

\subsection{Stem functions}

A  {\em stem function} is a (usually holomorphic) function $F$
defined on a {\em symmetric domain}%
\footnote{$D\subset\C$ is called {\em symmetric}
    iff $z\in D\iff \bar z\in D$.}
$D$ in $\C$
with values in 
$\Ac=A\tensor_\R\C$ satisfying $\overline{F(z)}=F(\bar z)$ where
we use the {\em complex conjugation} on $A\tensor_\R\C$.

For a stem function $F$ we define $(F^c)(z)=\left(F(z)\right)^c$,
i.e., we apply quaternionic conjugation pointwise.
Thus we obtain a {\em conjugation} on the algebra of stem functions
defined on a domain
$D\subset \C$. As before, we define norm and trace and obtain:
\[
\Nm(F)(z)=(FF^c)(z)=(F(z))(F^c(z))=(F(z))(F(z))^c=\Nm(F(z))
\]
and
\[
(\Tr  F)(z)=(F+F^c)(z)=F(z)+(F(z))^c=\Tr (F(z)).
\]

Globally defined slice regular functions are given as globally
convergent power series:
\[
f(q)=\sum_{k=0}^{+\infty} q^k a_k
\]
In this case
\[
(f^c)(q)=\sum_{k=0}^{+\infty} q^k a_k^c
\]

The space of slice regular functions on an axially
symmetric domain forms an associative $\R$-algebra with the
$*$-product as multiplication.

Hence $(\Tr f)=f+f^c$ and $\Nm(f)=f*f^c$.

Immediately from the construction we obtain:

\begin{proposition}\label{n-tr-stem}
  Let $f$ be a slice function and $F$ its associated stem function.

  Then 
  $\Nm(F)$, $\Tr(F)$ and $F^ c$ are the stem functions associated to
  $\Nm(f)$, $\Tr(f)$ and ~$f^c$.
\end{proposition}

{\em Remark.} In general we
  have $f^c(q)\ne\overline{f(q)}$.
Only for real points $q\in\R$ we have $f^c(q)=\overline{f(q)}$
and consequently $(\Nm f)(q)=\Nm(f(q))$ and
$(\Tr f)(q)=\Tr\left(f(q)\right)$.

\subsection{Compatibility}

Recall that conjugation for slice regular functions is {\em not}
just pointwise conjugation of the function values.

Therefore in general
\[
(\Tr f)(q)\ne \Tr(f(q)), (\Nm f)(q)\ne \Nm(f(q)).
\]

Let $B$ be a $\R$-sub algebra of an $\R$ algebra $A$
equipped
with an antiinvolution which preserves $B$.

Then for $x\in B$ the notions $\Tr(x)$ and $\Nm(x)$
defined with respect to this antiinvolution are the same
regardless whether we regard $x$ as an element of $B$ or as an
element of $A$.

As a consequence
\begin{itemize}
\item
  For $x\in\H$ the notions $\Nm(x)$, $\Tr(x)$ agree
  independent of whether
  we consider $x$ in $\H$ or in $\Hc$.
\item
  For an element  $x\in\H$ the notions $\Nm(x)$, $\Tr(x)$ agree
  whether we regard $x$ in $\H$ or as a constant slice regular function
  with value $x$.
\end{itemize}

\subsection{Other notions}

For $x\in\H$ the term
$\frac12\Tr(x)$ is often called {\em real part} of $x$, sometimes
denoted as $x_0$.

$N(x)$ is also called the {\em symmetrization} of $x$ and denoted as
$x^s$.

\section{An example}

\begin{warning}
  The conditions of the Main Theorem \ref{mainth}
  do {\em not} imply that for any $q\in \H$ we can find
an element $ \alpha\in Aut(\H)$ such that $f(q)=\alpha(g(q))$.
\end{warning}

\begin{example}
  Let $f(q)=I$ and let $g(q)=\cos(q)I+\sin(q)J$, i.e.,
  \[
  g(q)= \sum_{k=0}^{+\infty}  \left(  q^{2k}\frac{(-1)^k}{(2k)!}I
  +q^{2k+1} \frac{(-1)^k}{(2k+1)!}J
  \right).
  \]
where $I,J$ are any two orthogonal imaginary units in $\mathbb{H}$ with $K=IJ.$ \\
  We recall the classical identities
  \[
  \sin(it)=i\sinh(t),\quad \cos(it)=\cosh(t)\ \forall t\in\R.
  \]
  For $t\in\R$ we deduce
  \begin{align*}
  g(tJ)&=\cos(tJ)I+\sin(tJ)J\\
  &  =\cosh(t)I +\sinh(t)\underbrace{J^2}_{=-1}\\
  & =\cosh(t)I-\sinh(t).\\
  \end{align*}
  In particular,
  the ``real part''  $\frac 12\Tr\left(g(tJ)\right)$
  is non-zero if $\sinh(t)\ne 0$,
  i.e., if $t \ne 0$.
  Hence for $t\in\R^ *$,
  we have
  \[
  \Tr\left(g(tJ)\right)\ne 0 = \Tr(I)=\Tr\left(f(tJ)\right)
  \]
  Since $\Tr(\phi(q))=\Tr(q)$ for every $q\in\H$ and
  every automorphism $\phi$ of $\H$, it follows
  that there is 
  no automorphism of $\H$ mapping $g(tJ)$ to $f(tJ)=I$.

  On the other hand
  \begin{align*}
  g^c(q)&=-\cos(q)I-\sin(q)J\\
&  \quad\implies\quad\\
  (g*g^c)(q)&=-\left ( \cos(q)I+\sin(q)J \right) *
  \left ( \cos(q)I+\sin(q)J \right)\\
  &=
  -\left( \cos^2(q)\underbrace{I^2}_{=-1}+\sin(q)^2 \underbrace{J^2}_{=-1}+
  \sin(q)\cos(q)\underbrace{(IJ+JI)}_{=0}\right)\\
  &=
  \left(\underbrace{\cos^2(q)+\sin^2(q) }_{=1} \right)=1=f*f^c\\
  \end{align*}
  and
  \begin{align*}
  g(q)+g^c(q)&=\underbrace{\cos(q)I+\sin(q)J}_{g(q)}
+  \underbrace{-\cos(q)I-\sin(q)J}_{g^ c(q)}
=0\\
&=I+(-I)=
f(q)+f^c(q)\quad , \forall q\in\H. \\
  \end{align*}

  Thus $\Tr(f)=\Tr(g)=0$ and $\Nm(f)=\Nm(g)=1$.
    The stem functions
  $F,G:\C\to\H\tensor\C$ associated to $f$ and $g$ are:
  \[
  F(z)=1\tensor I,\quad G(z)=\cos(z)\tensor I+\sin(z)\tensor J
  \]
  Since $sin^2(z)+\cos^ 2(z)=1\ \forall z\in\C$,
  both $F$ and $G$ avoid the center $\R\tensor\C$ of $\H\tensor\C$.
  Hence the {\em central divisors} $\cdiv(F)$ and $\cdiv(G)$
  (as defined in \paragraph\ref{ss-cdiv}) are both empty.
  Therefore
  we have verified that in this example
  \[
  \Tr(f)=\Tr(g), \quad \Nm(f)=\Nm(g),\quad \cdiv(f)=\cdiv(g),
  \]
  but for some $q\in\H$ there is no $\phi\in\Aut(\H)$ with
  \[
  f(q)=\phi(g(q)).
  \]

On the other hand, our main result implies that there exists
a holomorphic map $\phi:\C\to\Aut(\Hc)$ such that the
corresponding stem functions $F$ and $G$ satisfy
$F(z)=\phi(z)(G(z))$.
  
\cite{AD1} and \cite{AD2} pointed out that there exists a
slice regular function $h$  such that $f*h=h*g$.
This implies that
\[
h^{-*}*f*h=g
\]
where $h$ is slice regular, but $h^ {-*}$ is possibly only semi-regular.

In contrast, our result implies the existence of such a {\em slice regular}
function $h$ which is {\em invertible} in the sense that $h^ {-*}$ is
likewise slice regular, and not only semi-regular.

To give an explicit example, let
\[
H(z)=\cos(z/2)-K\sin(z/2)
\]
Then
\begin{align*}
&H(z)^ {-1}\cdot F(z)\cdot H(z)\\
  =&
\left(\cos(\frac z2)+K\sin(\frac z2)\right)\cdot
I\cdot
\left(\cos(\frac z2)-K\sin(\frac z2)\right)\\
=& I\cos z+J\sin z=G(z)\\
\end{align*}
  \end{example}
\section{Another example}\label{an-ex}
Consider the stem functions
\begin{align*}
  F(z)&=I+zJ+\frac 12z^2K\\
  G(z)&=\left(1+\frac 12z^2\right)I
\end{align*}

We have
\begin{align*}
  &\Tr(F)=\Tr(G)=0\\
  &\Nm(F)=\Nm(G)=1+z^2+\frac 14z^4\\
  &\cdiv(F)=\{\},\quad  \cdiv(G)=1\{\sqrt2 i\}+1\{-\sqrt2 i\}\\
\end{align*}

$F$ and $G$ are not equivalent in our sense because $\cdiv(F)\ne\cdiv(G)$.
Therefore there does not exist a
stem function $H:\C\to\Hc^*$ such
that both $H$ and $H^ {-1}$ are holomorphic and
\[
F=H^{-1}\cdot G \cdot H
\]

In contrast, Altavilla and de Fabritiis do not need  a condition
on $\cdiv$ for their results in \cite{AD1} and \cite{AD2}.
Hence their results imply that there exists a {\em meromorphic}
stem function $H$ with
\[
F=H^{-1}\cdot G \cdot H
\]

Indeed, an explicit calculation shows that
\[
H(z)=I\left(2+\frac12z^2\right)+Jz+\frac12z^2K,\quad
H^{-1}(z)=-\frac{I\left(2+\frac12z^2\right)+Jz+\frac12z^2K}%
{\frac 12z^4+3z^2+4}
\]
is such a function.
\section{More preparations}

\subsection{The Clifford algebra $\R_3$}

The {\em Clifford algebra} $\R_3$ may be realized as the
associative $\R$-algebra generated by $e_1,e_2,e_3$ with the relations
\[
e_{jk}+e_{kj}=-2\delta_{jk} \quad j,k\in\{1,2,3\}
\]
where $\delta_{jk}$ is the Kronecker symbol and $e_{jk} =e_j \cdot e_k$, i.e.,
\[
\delta_{jk}=\begin{cases} 1 & \text{ if $j=k$}\\
0 & \text{ if $j\ne k$}\\
\end{cases}
\]

It contains the idempotents $\omega_+,\omega_-$ satisfying
$\omega_+\omega_-=0$, namely:
\[
\omega_+=\frac 12\left(e_1e_2e_3+1\right),\quad
\omega_-=\frac 12·\left(e_1e_2e_3-1\right)
\]

This yields a direct sum decomposition of $\R_3$:
\[
\R_3=\omega_+\H\oplus\omega_-\H
\]
where $\H\cong \R_2$ is embedded in $\R_3$
as the subalgebra generated by $e_1,e_2$.

See  e.g.~\cite{Cliff},\cite{Lam}, Chapter V
for more details on Clifford algebras.

\subsection{Invariants for $\R_3$}

For more clarity in this paragraph we use $\Nm_A$ resp.~$\Tr_A$ in order
to denote the norm and trace for a given algebra $A$.

The Clifford algeba $\R_3$ is isomorphic
(as $\R$-algebra with an anti-involution which we call conjugation)
to $\H\oplus\H$.

As a consequence we have

\begin{align*}
\Nm_{\R_3}(q_1,q_2)&=\Nm_{\H\oplus\H}(q_1,q_2)&=\left(\Nm_\H(q_1),\Nm_\H(q_2)\right)\\
\quad
\Tr_{\R_3}(q_1,q_2)&=\Tr_{\H\oplus\H}(q_1,q_2)&=\left(\Tr_\H(q_1),\Tr_\H(q_2)\right)\\
\end{align*}
for $q_1,q_2\in\H\oplus\H\cong\R_3$.

Similar
for the complexified algebra $\R_3\tensor_\R\C\cong\Hc\oplus\Hc$
and the corresponding function algebras.

{\em Caveat:} Since $(z,w)\mapsto (w,z)$ is an automorphism of
$\H\oplus\H$ (see \paragraph\ref{sect-aut-r3}),
these functions $\Nm$, $\Tr$ are not invariants.
They are only invariants up to changing the order of the
two components, i.e., they are invariants for the
connected component $\Aut^0(\H\oplus\H)$ which equals
$\Aut(\H)\times\Aut(\H)$ acting component-wise on $A=\H\oplus\H$.

\subsection{Slice preserving functions}

\begin{proposition}
  Let $A=\H$ and let
  $f:A\to A$ be a slice regular function with
  stem function $F:\C\to\Ac$.

  Then the following are equivalent:
  \begin{enumerate}
  \item
    $f=f^c$,
  \item
    $F=F^c$,
  \item
    $F(\C)\subset  \R\tensor_\R\C\subset A\tensor_\R\C=\Ac$.
  \item
    $f(\C_I)\subset\C_I$ for all $I\in\S=\{q\in A:q^2=-1\}$
    (with $\C_I=\R+I\R$).
  \end{enumerate}
\end{proposition}

\begin{definition}\label{slice-preserv}
If one (hence all) of these properties are fulfilled,
$f$ is called ``{\em slice preserving}''.
\end{definition}

\begin{proof}
  These equivalences are well known. $(iv)\iff(iii)$ follows from
  representation formula.
  $(i)\iff(ii)\iff(iii)$ by construction  of $(\ )^c$.
\end{proof}

\subsection{Slice regular functions for $A=\R_3$}

In \cite{GP} {\em slice regular functions} are considered
for real alternative algebras, e.g., Clifford algebras.
For generalities on Clifford algebras see e.g. \cite{Lam}, Ch.V.

A {\em quadratic cone}
$Q_A\subset A$ is defined
such that every $q\in Q_A$ is contained in the image of some
$\R$-algebra homomorphism from $\C$ to $A$.
While $Q_A=A$ for $A=\H$, we have $Q_A\ne A$ for
the Clifford algebra $A=\R_3\cong\H\oplus\H$.
More precisely (see \cite{GP}):
\begin{align*}
  Q_A  & = \R\cup\left\{x\in A: \Tr(x),\Nm(x)\in\R,4\Nm(x)>\Tr(x)^ 2\right\}
  \\
  & =  \left\{x\in A: \Tr(x),\Nm(x)\in\R\right\}
  \\
  &\cong\{(q_1,q_2)\in\H\oplus\H: \Tr_\H(q_1)=\Tr_\H(q_2),
  \Nm_\H(q_1)=\Nm_\H(q_2)\}
\end{align*}

(For $A=\R_3$ the condition $4\Nm(x)>\Tr(x)^ 2$ is automatically
satisfied for every element $x\in A\setminus\R$ with
$\Nm(x),\Tr(x)\in\R$.)

Given a symmetric domain $D\subset\C$, there is
the {``\em circularization''}
\[
\Omega_{\dom}=\{(x+yH_1,x+yH_2):x,y\in\R,H_1,H_2\in\S_{\H}, x+yi\in D\}.
\]
(with $\S_\H=\{J\in\H: J^2=-1\}$.)
Evidently $\Omega_D\subset Q_A$.

But due to the special direct sum nature of $\R_3$ we
may 
regard a larger set, namely :
\begin{equation}\label{def-omega_d}
W_D=\{(x_1+y_1H_1,x_2+y_2H_2):H_j\in\S_\H, x_j,y_j\in\R,
x_j+y_ji\in D\ \forall j\in\{1,2\}\}
\end{equation}

Note that $\Omega_D=Q_A\cap W_D$. Thus $\Omega_D$ relates to
  $W_D$ like the quadratic cone $Q_A$ to the whole algebra $A=\R_3$.

{\em Slice regular functions} as defined in \cite{GP}
are certain
functions defined on the
cone $Q_A$.
It is shown in \cite{GP} that these {\em slice regular functions}
correspond to {\em holomorphic stem functions}, i.e.,
holomorphic functions $F:\dom\to\Ac$
such that $\overline{F(z)}=F(\bar z)\ \forall z\in\C$.

We remark that there is a natural way
to extend a function
$f:\Omega_D\to \R_3$ to a function $\tilde f:W_D\to \R_3$.

Namely, we use the decomposition $\R_3\cong\H\oplus\H$
to represent $f$ as $f=(f_1,f_2)$ with $f_i:\Omega_D\to\H$
and define
\begin{align*}
&\tilde f(x_1+y_1H_1,x_2+y_2H_2)\\
\dfe
&\left(
f_1(x_1+y_1H_1,x_1+y_1H_2),
f_2(x_2+y_2H_1,x_2+y_2H_2)
\right)\\
\end{align*}

In this way the algebra of slice regular functions on $\Omega_D$
can be identified with a certain algebra of functions on $W_D$.

Let us now discuss the special case of globally defined functions
(i.e.~$D=\C$, $\Omega_D=Q_A$ and $W_D=A=\R_3$.)

In this case stem functions are
holomorphic functions with values in $\Ac$
which are defined on the whole of $\C$. Such a holomorphic
function may be defined by a  convergent power series
$\sum_{k=0}^{+\infty} z^ka_k$ with $a_k\in A$.
(A priori, $a_k\in \Ac$ for a holomorphic function
$F:\C\to\Ac$, but the condition
$\overline{F(z)}=F(\bar z)$ ensures that $a_k\in A$.)

Such a power series may also be regarded as power series
in a variable in $A$; in this case it defines a function
$f:A\to A$.
(It is easily shown that it is again globally convergent.) 

Thus there are bijective correspondances
between the following three classes of functions:

\begin{itemize}
  \item
    {\em Slice regular functions} $f:Q_A\to A$ as defined in \cite{GP}.
  \item
    {\em Entire functions}, i.e., functions
    $f:A\to A$ which are defined by globally convergent power series
    $\sum_{k=0}^{+\infty} q^k a_k$ ($a_k\in A$).
  \item
    {\em Stem functions}, i.e., holomorphic functions $F:\C\to \Ac$ with
    $\overline{F(z)}=F(\bar z)$.
\end{itemize}

In particular, slice regular functions $f:Q_A\to A$ extend naturally
to functions $f:A\to A$ defined by globally convergent power series.

  Thus for the special case of the Clifford algebra $\R_3$ there is no
  need to restrict the domain of definition to (subdomains of) the
  quadratic cone $Q_A$.

  Let $D,\Omega_D,W_D$ be as in the preceding subsection.

  We described above how to associate a function $\tilde f:W_D\to A\cong\H\oplus\H$
  to a given function $f:\Omega_D\to A$.

\section{Central divisors}
 \subsection{Divisors for vector-valued functions}

Normally, {\em divisors} are defined for holomorphic 
functions with values in $\C$. Here we extend this notion to holomorphic
maps from Riemann surfaces to higher-dimensional complex vector spaces.

\begin{definition}\label{def-vd}
  Let $f:X\to V=\C^n$ be a holomorphic map from a Riemann surface $X$
  to a complex vector space $V=\C^n$.
  Assume that $f$ is not identically zero.

  The {\em divisor} of $f$ is the divisor corresponding to the
  pull back of the ideal sheaf of the origin, i.e., for $f=(f_1,\ldots,f_n)$,
  $f_i:X\to\C$ we have $div(f)=\sum_{p\in X} m_p\{p\}$
  where $m_p$ denotes the minimum of the multiplicities $mult_p(f_i)$.
\end{definition}
\subsection{Central divisor for stem functions}\label{ss-cdiv}

In \cite{BW1}, Definition~3.1, we introduced the notion of a
{\em slice divisor}.
Here we will need a different notion of divisor.

Namely, we need a notion of divisor which measures how far the
function is from being {\em slice-preserving}.
This we call {``\em central divisor''}.
Here we define it for the quaternion case.

\begin{definition}\label{def-cd}
  Let $Z\cong\C$ be the center of $\Hc$.
  $\dom\subset\C$ be a symmetric domain, $F:\dom\to\Ac$ a holomorphic
  map. Assume $F(\dom)\not\subset Z$.

  The {\em central divisor} $\cdiv(F)$ is defined as the
  divisor
  (in the sense of Definition~\ref{def-vd})
  of the map from $\dom$ to $\Hc/Z$ induced by $F$.
\end{definition}

Let $(1,i,j,k)$ be the standard basis of $\H$ and
let $\WC$ denote the complex vector subspace
of $\Hc$ generated by $i\tensor 1$, $j\tensor 1$, $k\tensor 1$.

Then  $\Hc$ may be represented as the vector space direct sum
$\Hc=Z\oplus\WC$ where $Z\cong\C$ denotes the center.
and we can decompose $F:\dom \to\Hc$ as $F=(F',F''):\dom\to Z\times(\WC)$
and the central divisor $\cdiv(F)$ equals
$\sum_{p\in\dom} n_p\{p\}$ where $n_p$ denotes the vanishing order of $F''$
at $p$.

Let $F=(F',F'')$ be a stem function for a slice function $f$.
Then
$f$ is {\em slice-preserving}
if $F''\equiv 0$ (Definition~\ref{slice-preserv}).

Hence the assumption that $f$ is not slice preserving implies
that $F''$ does not vanish identically and we therefore may define
$\cdiv(F)$ as above.

\begin{example}
  Consider $F:\C\to\H\tensor_\R\C$ defined as
  \[
  F(z)=1\tensor z + i\tensor z^2(z-1) + j\tensor z^3(z-1)^2
  \]

  Then
  \[
  \cdiv(F)=
  2\cdot\{0\}+1\cdot \{1\}
  \]
\end{example}

{\em Caveat:} These central divisors do {\em not} satisfy the
usual functoriality:
\begin{example}
  Let
  \[
  F(z)=1+iz,\quad G(z)=1+j(1+z)
  \]
  Then $\cdiv(F)=1\cdot\{0\}$ and $\cdiv(G)=1\cdot\{-1\}$, but
  $\cdiv(FG)=\cdiv\left(1+iz+j(z+1)+k(z+1)z\right)$ is empty.
  Thus
  \[
  \cdiv(FG)\ne\cdiv(F)+\cdiv(G).
  \]
\end{example}

\subsection{Central divisor for slice functions}

We simply define the {\em central divisor} for a given
slice regular function as the central divisor
of the corresponding stem function
(as defined in Definition~\ref{def-cd}).

\subsection{Central divisors for $\R_3$} 
The {\em central divisor} $\cdiv$
of a slice regular function $f:\R_3\to\R_3$
is defined via the stem function $F$. Given a stem function
$F:\dom\to\Ac\cong\Hc\oplus\Hc$, we define 
\[
\cdiv(F)=(\cdiv F_1,\cdiv F_2)
\]
where $F_1,F_2$ are the components of the stem function
$F$ with  respect to the decomposition of $\R_3$ into a
direct sum of two copies of $\H$.

What we need, is to treat $\R_3$ consequently as product,
i.e., we regard slice regular functions from (a domain in)
$\R_3$ to $\R_3$ as a pair of the quaternionic slice regular functions
and {\em define} the {``central divisor''} of a slice regular
function $f=(f_1;f_2)$ as
\[
\cdiv(f)=\cdiv(f_1;f_2)\dfe\left(\cdiv(f_1),\cdiv(f_2)\right)
\]

\section{Strategy}
Let $A$ be a $\R$-algebra,
$\Ac=A\tensor_\R \C$ its complexification and
$\GC$ the automorphism group of the complex algebra $\Ac$.

For certain%
\footnote{``Stem functions''}
 holomorphic maps $F,H:\C\to\Ac$
we want to show the following statement:

{\em If for every $x\in\C$ there exists an element $g\in\GC$ such that
  $F(x)=g\left(H(x)\right)$, then there exists a holomorphic map
  $\phi:\C\to\GC$ with $F(x)=\phi(x) \left(H(x)\right)\ \forall x$.
}

This amounts to find a section for a certain projection map, namely
$\pi:V\to\C$
with
\[
V=\{(x,g)\in\C\times\GC: F(x)=g\left(H(x)\right) \}
\]
and $\pi(x,g)=x$.

First we discuss the locus where $F$ and $H$ assume zero as value.

Outside this zero locus $D$, i.e.,
restricted
to $V_0=\pi^{-1}(\C\setminus D)$, the map $\pi$
has some nice properties if $A\cong\H$:
\begin{enumerate}
\item
  all $\pi$-fibers are of the same dimension,  homogeneous and moreover biholomorphic
  to a Lie subgroup of $\GC$.
\item
  There are everywhere local holomorphic sections.
\end{enumerate}

Moreover, for $A=\H$ there is a Zariski open subset $\Omega$
in $\C$
on which $\pi$ is a locally holomorphically trivial fiber bundle with
$\C^*$ as fiber. 

We show that there exists a global holomorphic section on $\Omega$.

We will use
strongly that the generic isotropy group
(i.e.~$\C^ *$) is both commutative and one-dimensional.

\section{Main Result for $\R_3$}

\begin{theorem}\label{thm-r3}
  Let $\dom\subset\C$ be a symmetric domain, $\Omega_{\dom}\subset \R_3$
  as defined
  above ( equation\eqref{def-omega_d} ), $f,h:\Omega_D\to \R_3$
  slice regular functions as defined above.
  
  Let $F,H:\dom\to \Ac$ be the associated
  {\em stem functions}.

Let $\GC$ denote the
  connected component of the neutral element of
  the group of $\C$-algebra automorphisms of
  $\Ac=\R_3\tensor_{\R}\C$, i.e.,
  $\GC=\Aut(\Hc)\times\Aut(\Hc)$.
  
  Let $\Nm$, $\Tr$ and $\cdiv$ be defined as in subsection 5.2. \\
  Then the following are equivalent:

  \begin{enumerate}[label=(\arabic*)]
    \item
    There exists a holomorphic map
    $\phi:\dom\to \GC$ with
     \[
     F(z)=\phi(z)\left( H(z)\right)\ \forall z\in\dom
     \]
   \item
    There exists a holomorphic map
    $\alpha:\dom\to \Ac^*\cong\Hc^*\times\Hc^*$ with
     \[
     F(z)=\alpha(z)\left( H(z)\right)\left(\alpha(z)\right)^{-1}\
     \forall z\in\dom
     \]
   \item
      $\cdiv(F)=\cdiv(H)$, $\Tr (F)=\Tr (H)$ and $\Nm(F)=\Nm(H)$.
  \end{enumerate}

\end{theorem}

\begin{proof}%
Using $A\cong\H\oplus\H$ and $\GC^0\cong \Aut(\Hc)\times\Aut(\Hc)$,
  the equivalence of the existence of such a map $\phi:\dom\to\GC$
  with the condition
  \[
  \cdiv(F)=\cdiv(H), \Tr (F)=\Tr (H) \text{ and }\Nm(F)=\Nm(H)
  \]
  follows from the respective result for quaternions
  (Theorem~\ref{mainth}, proved in \paragraph\ref{pf-main}).
\end{proof}

\section{Local equivalence}

\begin{SpecAss}\label{SpecAss}
Let $G$ be a connected complex Lie group acting holomorphically on
  a complex manifold $X$ such that all the orbits have the same
  dimension $d$.

  Let
  \[
  F,H:\Delta=\{z\in\C:|z|<1\}
  \to X
  \]
  be holomorphic maps such that for every
  $z\in\Delta$ there exists an element $g\in G$ (depending on $z$,
  not necessarily unique)
  with $F(z)=g\cdot H(z)$.

  Let
  \[
  V=\{(z,g)\in\Delta\times G: F(z)=g\cdot H(z)\}
  \]
\end{SpecAss}

\begin{lemma}\label{fiber-iso}
  Under the above assumptions (\ref{SpecAss})
  for every $t\in\Delta$ there is
  a biholomorphic map between the fiber
  \[
  V_t=\{(z,g)\in V: z=t\}
  \]
  and the isotropy group
  \[
  G_{F(t)}=\{g\in G: g\cdot F(t)=F(t)\}
  \]
\end{lemma}

\begin{proof}
  Fix $t\in\Delta$ and choose a point $(t,g_0)\in V_t$.
  Note that
  \[
  F(t)=g_0\cdot H(t)
  \]
  because $(t,g_0)\in V_t\subset V$.

  We define a map
  \[
  \zeta:G_{F(t)}\stackrel{\sim}{\longrightarrow} V_t
  \]
  as
  \[
  \zeta(g)=(t,g\cdot g_0)
  \]

  We claim that this is a biholomorphic map from $G_{F(t)}$
  to $V_t$.
  First we verify that $\zeta(g)\in V_t$. Indeed,
  $F(t)=g_0\cdot H(t)$
  (as seen above) and $g\cdot F(t)=F(t)$, 
  because $g\in G_{F(t)}$.

  Combined these facts imply
  \[
  F(t)=g\cdot \underbrace{g_0\cdot H(t) }_{F(t)}
  \ \implies\ \zeta(g)=(t,gg_0)\in V
  \]
  $\zeta$ is  obviously injective. Let us check surjectivity.
  Let $(t,p)\in V_t$. Then $F(t)=p\cdot H(t)$.
  Recall that $F(t)=g_0\cdot H(t)$. It follows that
  \[
  (p\cdot g_0^{-1})F(t)=(p\cdot g_0^{-1})g_0\cdot H(t)=pH(t)=F(t)
  \]
  Hence
  \[
   p\cdot g_0^{-1}\in G_{F(t)}.
   \]
   We claim that $\zeta\left(p\cdot g_0^{-1}\right)=(t,p)$.
   Indeed
   \[
   \zeta\left(p\cdot g_0^{-1}\right)=
   \left(t,\left(p\cdot g_0^{-1}\right)\cdot g_0\right)
   =(t,p).
   \]
  
  Therefore 
  \[
  \zeta:G_{F(t)}\stackrel{\sim}{\longrightarrow} V_t
  \]
  is biholomorphic.
\end{proof}
  
\begin{lemma}\label{v-smooth}
Under the above assumptions  (\ref{SpecAss})
$V$ is smooth.
\end{lemma}

\begin{proof}
    Consider
    \[
    Z:=\left\{ (z,x,x):z\in\Delta,x\in X\right\}.
    \]
    Note that $Z$ is a submanifold of $\Delta\times X\times X$.

    Define
    \[
    \phi:\Delta\times G\to \Delta\times X\times X,
    \quad \phi(z,g)=\left(z,F(z),g\cdot H(z)\right)
    \]
    and observe that $V=\phi^{-1}(Z)$.

    In order to verify the smoothness of $V$ it suffices
    to show that $D\phi$ has everywhere the same rank.
    
    Let $(x,g)\in \Delta\times G$ and consider
    \[
    D\phi_{(x,g)}:T_{(x,g)}\left(\Delta\times G\right)
    \to T_{\phi(x,g)}\left( \Delta\times X\times X\right).
    \]
    We observe that
    \[
    T_{(x,g)}\left(\Delta\times G\right)\cong \left(T_x\Delta\right)\times Lie(G).
    \]
    From $\phi(z,g)=(z,F(z),g(H(z)))$ we infer that
    \begin{align*}
    &(v,w)\in\ker D\phi_{(x,g)}\subset
      \left(T_x\Delta\right)\times Lie(G)\\
      \iff &
      v=0\text{ and }w\in\ker D\zeta\ \text{ for }
      \zeta:g\mapsto g\left(H(z)\right)
      \\
    \end{align*}
    Here $\zeta:G\to X$ is the orbit map $g\mapsto g\left(H(z)\right)$.
    Standard theory of transformation groups implies
    that $w\in\ker D\zeta$ iff $w\in T_gG\cong\Lie(G)$
    is contained in the Lie subalgebra of the isotropy group
    of the $G$-action at $g\left(H(z)\right)$.

    By assumption all the $G$-orbits in $X$ have the same dimension $d$,
    Hence every isotropy group is of dimension $\dim(G)-d$.
    It follows that
    \[
    \dim\ker D\phi_{z,g}=\dim(G)-d\ \ \forall z,g
    \]
    Thus $\phi$ is a map of constant rank and
    $V=\phi^{-1}(Z)$ is smooth.
    \end{proof}

\begin{proposition}\label{loc-eq}
 Under the above assumptions,
 there exists $0<r<1$ and a holomorphic map
  \[
  \phi:
  \Delta_r=\{z:|z|<r\}\to G
  \]
  such that
  \[
  F(z)=\phi(z)\left( H(z)\right)\ \forall z\in\Delta_r.
  \]
\end{proposition}

\begin{proof}
  Let $g_0\in G$ such that $F(0)=g_0\left(H(0)\right)$.
  We may replace $H(z)$ by the function $\tilde H(z)=g_0\left(H(z)\right)$.
  In this way we see: 
  There is no loss of generality in assuming $F(0)=H(0)$.

  We may replace $G$ by its universal covering and therefore assume
  that $G$ is simply-connected. Then we can use the fact that
  every simply-connected complex Lie group is a Stein manifold
  (see \cite{MM}). Hence we may assume that $G$ is a Stein
  manifold. 
  
  Recall that
  \[
  V=\{(z,g)\in\Delta\times G: F(z)=g\left( H(z)\right)\}.
  \]
  Let $\pi:V\to\Delta$ be the natural projection and
  $V_t=\pi^{-1}(t)$ (for $t\in\Delta$).
  Due to Lemma~\ref{fiber-iso}
  there is a bijective map from $V_t$  to the isotropy group of
  the $G$-action at $F(t)$. Since all the $G$-orbits are
  assumed to have the same dimension $d$, each isotropy
  group has the dimension $\dim(G)-d$.
  Therefore
  all the fibers of $\pi:V\to\Delta$ have the
  same dimension (namely $\dim(G)-d$).

    Due to Lemma~\ref{v-smooth} the complex space $V$ is smooth.

    We consider the relative symmetric product $\Sm$, i.e.,
    $\Sm$ is the quotient of
    \[
    \{(z;v_1,\ldots,v_m)\in \Delta\times V^m:\pi(v_i)=z, \,\, \forall i\}
    \]
    under the natural action of the symmetric group $S_m$
    permuting the components of $V^m$.

    $V$ may be embedded into $\Sm$ diagonally as
    \[
    V\ni v=(z,g)\stackrel{\delta}{\mapsto} [(z;v,\ldots,v)]\in\Sm.
    \]
    Note that $V$ is Stein as a closed analytic
    subspace of the Stein manifold $\Delta\times G$.
    It follows that
    \[
    \{(z;v_1,\ldots,v_m)\in \Delta\times V^m:\pi(v_i)=z, \,\, \forall i\}
    \]
    is likewise Stein.
    Thus $\Sm$ is the quotient of a Stein space
    by a finite group (namely $S_m$).
    Hence $\Sm$ is Stein and therefore
    admits
    an embedding $j:\Sm\hookrightarrow\C^N$ for some $N\in\N$.
    Recall that $\Sm$ is a quotient of a subspace of $\Delta\times V^ m$ by a
    $S_m$-action which is trivial on the first factor $\Delta$. The
    natural projection from $\Delta\times V^ m$ onto its first factor
    thus yields a natural map $p:\Sm\to\Delta$.
    
    Now we define an embedding $\xi:\Sm\to \Delta\times\C^N$
    as
    \[
    \xi:w\mapsto (p(w),j(w))
    \]
    We obtain a commutative diagram
    \begin{center}
    \begin{tikzpicture}
      \matrix (m)
              [
                matrix of math nodes,
                row sep    = 3em,
                column sep = 4em
    ]
    {
      V & \Sm & \Delta\times\C^N \\
      \Delta &  \Delta &\Delta \\
    };
    \path
    (m-1-1) edge [->] node [above] {$\delta$} (m-1-2)
    (m-1-2) edge [->] node [above] {$\xi$} (m-1-3)
      (m-2-1) edge  [double distance=3pt] (m-2-2)
      (m-2-2) edge  [double distance=3pt] (m-2-3)
    (m-1-1) edge [->] node [left] {$\pi$} (m-2-1)
    (m-1-2)  edge [->] node [right] {$p$} (m-2-2)
    (m-1-3)  edge [->] node [right] {$pr_1$} (m-2-3)
    ;
    \end{tikzpicture}
    \end{center}

    Here $pr_1:\Delta\times\C^ N\to\Delta$ is the projection
    to the first factor.
    
    Because $V$ is smooth
    (see Lemma~\ref{v-smooth}), the
    ``tubular neighbourhood theorem'' (see \cite{F2},
    Theorem 3.3.3.) implies the existence of a holomorphic retraction of an
    open neighborhood $W$ of $Y=\xi(\delta(V))$ in
    $\Delta\times\C^N$ onto $Y$.

    This can be done in a relative way, over $\Delta$,
    cf.~\cite{F}, Lemma 3.3; \cite{F2}, Theorem 3.3.4.

    Thus there is a holomorphic map $\rho:W\to Y$
    with $\rho|_Y=id_Y$ and $pr_1(w)=pr_1(\rho(w))$ for all
    $w\in W$.
\[
    \begin{tikzcd}
      \xi(\delta(V)) \arrow[r, equal] &
      Y \arrow[dr]
      \arrow[r,shift left,hookrightarrow] &
      W \arrow[l,"\rho",shift left]
      \arrow[r,hookrightarrow,"open"]
      \arrow[d]&
      \Delta\times\C^ N
      \arrow[ld]\\
      &&\Delta&\\
    \end{tikzcd}
\]
    Let $C$ be a (local) smooth complex curve in
    $V$ through the point $(0,e)$ which is not contained
    in the fiber $V_0=\pi^ {-1}(0)$.
    Now $C\cap V_0$ is a discrete subset of $C$.
    Hence by shrinking $C$ we may assume that
    with
    \[
    C\cap V_0=\{(0,e)\}
    \]

    Now $\pi|_C:C\to\Delta$ is a non-constant holomorphic map between
    one-dimensional complex manifolds.
    
    By one-dimensional complex analysis, after appropriately shrinking
    $\Delta$ and $C$, the projection map from $C$ onto $\Delta$
    is a finite ramified covering of some degree $m\in\N$.

    Thus we obtain a local multisection, i.e., there is an
    open neighborhood $\Delta'$ of $0$ in $\Delta$
    and a closed analytic subset $C'\subset \pi^{-1}(\Delta')$ such that
    $\pi|_{C'}:C'\to\Delta'$ is a finite ramified covering. Let $m$
    denote the degree of $\pi|_{C'}$.

    To each point $t\in\Delta'$ we associate the finite
    space $\pi^{-1}(t)\cap C'$.
    This is a finite subspace of $\pi^{-1}(t)\subset V$
    of degree $m$ which is mapped by $\pi$ into one single point of $\Delta$.
    Such a finite subset
    corresponds to a point in the relative symmetric product $\Sm$.
    Therefore our multisection yields a section $s:\Delta'\to \Sm$.
    
    We recall that we assumed $C\cap V_0=\{(0,e)\}$.
    Hence
    \[
    s(0)=\delta(0,e)\ \implies\ \xi(s(0))\in Y\subset W
    \]
    Now $W$ is open, and $(\xi\circ s)(0)\in W$.
    Let $\Delta''=(\xi\circ s)^{-1}(W)$.

    We compose $\xi\circ s:\Delta''\to W$ with the holomorphic retraction
    \[
    \rho:W\to Y = \xi(\delta(V)).
    \]
    Since $\xi\circ\delta:V\to Y$ is an isomorphism,
    there is a unique holomorphic map
    $\sigma:\Delta''\to V$ satisfying
    \[
    \xi\circ\delta\circ\sigma=\rho\circ\xi\circ s,
    \]
    namely $\sigma=(\xi\circ\delta)^{-1}\circ\rho\circ\xi\circ s$.
    
    By construction, this map $\sigma$ is a section of $\pi:V\to\Delta$
    on $\Delta''$, i.e., a holomorphic map
    $\sigma:\Delta''\to V$ with $\pi\circ\sigma(z)=z, \ \forall z\in\Delta''$.

    Recall that $V$ is defined as
    \[
      V=\{(z,g)\in\Delta\times G: F(z)=g\left( H(z)\right)\}.
    \]
    Therefore a section $\sigma$ is given as
    \[
    \sigma:z\mapsto (z,\phi(z))
    \]
    where $\phi:\Delta''\to G$ is a map which fulfills
    \[
    F(z)=\phi(z)\left( H(z)\right)\ \forall z\in\Delta''.
    \]
    This yields the statement of the proposition if we choose $r\in(0,1)$
    such that $\Delta_r\subset\Delta''$.
\end{proof}

\begin{remark}
  It is essential to assume that all the orbits have the same
  dimension. For example, let $G=\C^*$ act on $X=\C$ as usual,
  let $F(z)=z^2$ and $H(z)=z^3$.
  Then $\forall z, \,\, \exists \,\, \lambda\in G: \lambda z^3=z^2$,
  but there is no holomorphic function $\phi:\Delta\to G$ with
  $\phi(z)z^3=z^2$.
\end{remark}
\section{Automorphisms of  Quaternions}

\subsection{Automorphisms of $\H$}

We will need the following classical result
(see e.g.~\cite{CS},\paragraph 6.8). It is based on
the fact that $\H$ is a central simple $\R$-algebra and the
Skolem-Noether theorem (see e.g.~\cite{LF},\paragraph29) which
states that every automorphism of a central simple algebra is inner.

\begin{proposition}\label{aut-h}
  Every ring automorphism of the $\R$-algebra of quaternions $\H$
  is already a $\R$-algebra automorphism (thus $\R$-linear and continuous)
  and preserves the scalar product on $\H$ defined as
  \[
  \left<x,y\right>=\frac 12\left(x\bar y+y\bar x\right)
  \]
  Let $G$ be the 
  group of ring automorphisms of $\H$.
  
Then $G$ is isomorphic to $SO(3,\R)$,
acting trivially on the center $\R$ and by the standard action of $SO(3,\R)$
on $\R^3,$ the orthogonal complement $W$ of $\R,$ if we identify $W$ with
$\R^3$ using the standard basis $i,j,k$ of $W$.

For $q\in\H$ let $G_q$ be its isotropy group,
  i.e., $G_q=\{g\in G=\Aut(\H):g(q)=q\}$.

  Then $G=G_q$, if $q$ is in the center $\R$ of $\H$ and
  $G_q\cong S^1\cong SO(2,\R)$ for $q\not\in\R$.
\end{proposition}

\begin{remark}
  We have $\Gc\cong PSL_2(\C)\cong SO(3,\C)$ acting on a three-dimensional
  vector space preserving a non-degenerate bilinear form. This representation
  must be isomorphic to the adjoint representation of $PSL_2(\C)$ on
  its Lie algebra
  $Lie(PSL_2(\C))$ which preserves the {\em Killing form}
  $Kill(x,y)=Trace(ad(x)ad(y))$.
  Let $v\in Lie(PSL_2(\C))\setminus\{0\}$.
  Then $Kill(v,v)=0$ if the one-parameter subgroup $\exp(v\C)$ is
  unipotent, while $\exp(v\C)\cong\C^*$ if $Kill(v,v)\ne 0$.
  The isotropy group of $\Gc$ at $v\in Lie(\Gc)$  is the centralizer
  of $\exp(v\C)$. For $\Gc=PSL_2(\C)$ and $v\ne 0$
  this is always $\exp(v\C)$ itself.
  Hence the isotropy group of $\Gc$ acting on $Lie(\Gc)$ at an element $v\ne 0$
  is isomorphic to $\C^*$ if
  \[
  Kill(v,v)\ne 0
  \]
  and isomorphic
  to $\C$ if $Kill(v,v)=0$.
  Similarily for  $\Gc$ acting on $W\tensor_\R\C$, replacing $Kill(\ ,\ )$
  by $B(\ ,\ )$.
\end{remark}

As a consequence one obtains:

  \begin{corollary}\label{n-tr-auto}
    Let $p,q\in \H\setminus\{0\}$.

    Then the following conditions are equivalent:
    \begin{itemize}
    \item
      $\Nm(p)=\Nm(q)$ and $\Tr(p)=\Tr(q)$.
    \item
      There is an $\R$-algebra automorphism $\phi$ of $\mathbb{H}$ such that
      $\phi(p)=q$.
    \end{itemize}
  \end{corollary}

\section{Automorphisms of $\R_3=\H\oplus\H$}\label{sect-aut-r3}

$\R_3$ denotes the Clifford algebra associated to $\R^3$ endowed with a
positive definite quadratic form.
As an $\R$-algebra, $\R_3\cong\H\oplus\H$.
This direct sum structure allows an easy determination of the automorphism group
of $\R_3$.

\subsection{Algebraic Preparation}
\begin{lemma}\label{prod-auto}
  Let $R$ be a (possibly non-commutative) ring with $1$ and without zero-divisors.
  Then every ring automorphism of $A=R\oplus R$ is either of the form
  \[
  (x,y)\mapsto (\phi(x),\psi(y))\quad \phi,\psi\in Aut(R)
  \]
  or of the form
  \[
  (x,y)\mapsto (\psi(y),\phi(x))\quad \phi,\psi\in Aut(R)
  \]
\end{lemma}

\begin{proof}
  Since $R$ has no zero-divisors, the only idempotents are $0$ and $1$,
  i.e.,
  \[
  \{x\in R: x^2=x\}=\{0,1\}
  \]
  As a consequence,
  \[
  I=  \{z=(x,y)\in R\times R: z^2=z\}=\{0,1\}\times\{0,1\}
  \]
  Every ring automorphism $\alpha$ of $A$ must stabilize $I$ and
  fix $0=(0,0)$ and $1=(1,1)$.
  Hence either $\alpha$ fixes $(0,1)$ and $(1,0)$, or
  \[
  (0,1)\stackrel{\alpha}{\longrightarrow}
  (1,0)\stackrel{\alpha}{\longrightarrow}
  (0,1)
  \]
  In the first case, $\alpha$ stabilizes both
  $(0,1)A=\{0\}\times R$ and
  $(1,0)A= R\times\{0\}$,
  which implies that $\alpha(x,y)=(\phi(x),\psi(y))$ for some
  $\phi,\psi\in Aut(R)$.

  Similarily in the second case
  \[
  (x,y)\mapsto (\psi(y),\phi(x))
  \]
  for some $\phi,\psi\in Aut(R)$.
\end{proof}

\subsection{Description of $Aut(\R_3)$}
As a consequence of Lemma~\ref{prod-auto}
and the description of the automorphisms of $\H$, we obtain:

\begin{proposition}
  The automorphism group of the Clifford algebra
  $\R_3\cong\H\oplus\H$ is generated
  by $(z,w)\mapsto(w,z)$ and $SO(3,\R)\times SO(3,\R)$
  where $SO(3,\R)$ acts as the group of all orientation preserving
  ortho\-go\-nal linear transformations of the imaginary parts of the
  factors of the product $\H\times\H$.

  The automorphism group of the complex algebra
  $\R_3\tensor\C$ is generated
  by $(z,w)\mapsto(w,z)$ and $SO(3,\C)\times SO(3,\C)$.
\end{proposition}

  \subsection{Isotropy groups}

    Using the above description of the full automorphism group
  and the product structure $\R_3\cong\H\oplus\H$ it is
  easy to determine the isotropy groups of an element
  $(q_1,q_2)$:

  \begin{itemize}
  \item
    The isotropy group contains all automorphisms of the
    form $\phi:(x_1,x_2)\mapsto(\phi_1(x_1),\phi_2(x_2))$
    with $\phi_i\in\Aut(\H)$ and $\phi_i(q_i)=q_i$ ($i\in\{1,2\}$).
  \item
    If there is an automorphism $\alpha\in\Aut(\H)$ with
    $\alpha(q_1)=q_2$, then the isotropy group contains in
    addition all automorphisms of the form
    $\phi:(x_1,x_2)\mapsto(\phi_1(x_2),\phi_2(x_1))$
    with $\phi_i\in\Aut(\H)$ and
    \[
    (q_1,q_2)=(\phi_1(q_2),\phi_2(q_1))
    \]
    In view of $\alpha(q_1)=q_2$ the above
    equation is equivalent to
    \begin{align*}
      &(\alpha\circ\phi_1)(q_2)&=&\alpha(q_1) &=&q_2\\
      &(\alpha^{-1}\circ\phi_2)(q_1)&=&\alpha^{-1}(q_2)&=&q_1\\
    \end{align*}
    Therefore $\alpha\circ\phi_1$ must be in the isotropy group
      of $q_2$ and $\alpha^{-1}\circ\phi_2$ in the isotropy group of $q_1$.
      
  \end{itemize}
  
    \section{Orbits in the complexified algebra}

  The proposition below is principally applied to the situation,
  where $A\cong\H$ and $A=\R\oplus V$ as vector space,
  $V$ being the subspace of totally imaginary elements.

  In fact, we have seen
  that in this case $\Aut(A)$ acts trivially
  on $\R$ and by orthogonal transformations on $V$ such that
  every sphere centered at the origin in $V$ is one orbit.
  
  \begin{proposition}\label{VC-orbits}
  Let $V=\R^n$, and let $G$ be a connected real Lie group
  acting by orthogonal linear transformations on $V$ such that
  the unit sphere $S=\{v\in\R^n:||v||=1\}$ is a $G$-orbit.

  Let $\VC=V\tensor_\R\C$. Let $B$ denote the $\C$-bilinear form
  on $\VC$ extending the standard euclidean scalar product on $V=\R^n$.
  
  Let $\GC$ be the smallest complex Lie subgroup of $GL(\VC)$
  containing $G$. Then the $\GC$-orbits in $\VC$ are the
  following:

  \begin{itemize}
    \item
    $H_\lambda=\{v\in\VC:B(v,v)=\lambda\}$ for $\lambda\in\C^*$.
  \item
    $H_0=\{v\in\VC:B(v,v)=0\}\setminus\{0\}$
  \item
    $\{0\}$.
  \end{itemize}
\end{proposition}

\begin{proof}
  First we observe that $B(\ ,\ )$ is invariant under the $\GC$-action,
  because $G$ acts by orthogonal transformations.

  Since $G$ acts transitively on $S$, and $G$ acts linearly on $V$,
  the $G$-orbits in $V$ are precisely $\{0\}$ and the spheres
  $S_r=\{v\in V:||v||=r\}$ ($r>0$).

  Let $v=u+\iota w\in\VC$, $u,w\in V$. If $v\ne 0$, then
  $(u,w)\ne (0,0)$. If $u\ne 0$, then the projection onto the
  real part yields a map from the $G$-orbit through $v$
  onto a sphere in $\R^n$, which implies that the real tangent space
  of $Gv$ contains at least $n-1$ $\C$-linearly independent
  tangent vectors. It follows that the $\GC$-orbit through $v$ is
  of complex dimension $\ge n-1$.

  On the other hand the complex hypersurfaces
  \[
  C_\lambda=\{v\in\VC: B(v,v)=\lambda\}
  \]
  are irreducible, complex $(n-1)$-dimensional and $\GC$-invariant.
  This implies the assertion.
\end{proof}

\begin{corollary}\label{eq-iso-dim}
  Let $A$ be a finite-dimensional $\R$-algebra with $\R$ as center.
  Let $A=\R\oplus V$ as vector space and let $G$ be a real Lie group
  acting trivially on $\R$ and by orthogonal linear transformations on $V$.
  Assume that $G$ acts transitively on the unit sphere of $V$.

  Let $\GC$ be the smallest complex Lie subgroup of
  $GL(A\tensor_\R\C)$ containing $G$.

  Then all the $\GC$-orbits in
  $\left(V\tensor_\R\C\right)\setminus\{0\}$
  are complex hypersurfaces. In particular, they all
  have the same dimension.
\end{corollary}

\begin{corollary}\label{dim-c-isotropy}
  Under the assumptions of Corollary~\ref{eq-iso-dim}
  for every point $q\in\left(V\tensor_\R\C\right)\setminus\{0\}$
  the isotropy group $I$ of the $\GC$-action at $q$ satisfies
  \[
  \dim_\C I+\dim_\R V-1=\dim_\C(\GC)
  \]
  In particular, we have $\dim_\C(I)=1$ for $A\cong\H$.
\end{corollary}

\begin{proof}
  The $\GC$-orbit through  $q$ is a complex hypersurface
  (Corollary~\ref{eq-iso-dim}). Hence
  its complex dimension equals $\dim_\C(V\tensor_\R\C)-1$
  and therefore
  \[
  \dim_\C(\GC)=\dim_\C I+\dim_\C(V\tensor_\R\C)-1
  \]
  Now  $\dim_\R(V)=\dim_\C(V\tensor_\R\C)$.
  This yields  the assertion.
\end{proof}

\begin{corollary}\label{omegaconj}
  Under the same assumptions, there is a Zariski open subset,
  namely
  $\Omega\subset A\tensor_\R\C\stackrel{\zeta}
  \sim\C\oplus \left(V\tensor_\R\C\right)$
  defined as $\Omega=\{q:\zeta(q)=(x,v), B(v,v)\ne 0\}$
  such that all the isotropy groups
  of $\GC$ at points in $\Omega$ are conjugate.
\end{corollary}

\begin{proof}
  By assumption, $A=\R\oplus V$ as vector space and correspondingly
  \[
  A\tensor_\R\C\cong\C\oplus 
  \underbrace{\left(V\tensor_\R\C\right)}_{\VC}
  \]
  With respect to this vector space sum decomposition we define
  \[
  \Omega=\{r+v:r\in\C,v\in \VC, B(v,v)\ne 0\}
  \]
  Due to Proposition~\ref{VC-orbits},
  the $\GC$-orbit through a given point $x=r+v\in\Omega$
  is
  \[
  \{r+w,w\in \VC, B(w,w)=B(v,v)\}.
  \]
  The isotropy groups in the same orbit are clearly conjugate.
  Let $x=r+v,y=s+w\in\Omega$
  (with $r,s\in\C$, $v,w\in\VC$)
  be in different orbits.
  Choose $\lambda\in\C^*$ such that $B(w,w)=B(\lambda v,\lambda v)$.
  Now the  isotropy group at $r+v$ agrees with the isotropy group
  at
  $\lambda(r+v)=\lambda r+\lambda v$, because the group action is linear.
  Since $\GC$ acts trivially on the center $\C$ of $A\tensor_\R\C$,
  the isotropy groups at $\lambda r+\lambda v$ and $s+\lambda v$
  coincide.
  Finally, due to $B(w,w)=B(\lambda v,\lambda v)$ the elements
  $s+\lambda v$ and $s+w$ are in the same $\GC$-orbit.
  Consequently, their
  isotropy groups are conjugate.
  Combined, these facts imply that the isotropy  group at $r+v$ and
  $s+w$ are conjugate.
\end{proof}

\begin{corollary}\label{aut-if-n-tr}
  Let $A=\H$  and $\Ac=A\tensor_{\R}\C$.

  Let $p,q\in\Ac\setminus\{0\}$.
  Then the following properties are
  equivalent:
  \begin{enumerate}
  \item
    $\Tr (p)=\Tr (q)$ and $\Nm(p)=\Nm(q)$.
  \item
    There is an automorphism $\phi\in \Aut(\Ac)$ such that
    $\phi(p)=q$.
  \end{enumerate}
\end{corollary}

\begin{proof}
  By construction, we have $\Nm(x)=x(x^c)=B(x,x)$ for all $x\in\Ac$
  and $x\mapsto \frac 12tr(x)$ equals the projection of $x$ to $\C$
  with respect to the direct sum decomposition $\Ac=\C\oplus W$.

  Let $p=p'+p''$, $q=q'+q''$ with $p',q'\in\C$, $p'',q''\in W$.
  Since $\Aut(\Ac)$ acts trivially on $\C$ and by linear,
  $B$-preserving transformations on $W$, $(ii)$ implies that
  $p'=q'$ and $B(p'',p'')=B(q'',q'')$ which in turn implies
  $\Tr (p)=\Tr (q)$, $\Nm(p)=\Nm(q)$.

  Conversely,
  \[
  (  \Tr (p)=\Tr (q)) \land ( \Nm(p)=\Nm(q)))
  \implies
  (  p'=q') \land\left( B(p'',p'')=B(q'',q'')\right)
  \]
  and the latter condition implies
  $\exists\,\,\phi\in \Aut(\Ac):\phi(p)=q$ due to
  Proposition~\ref{VC-orbits}.
\end{proof}

\section{Constructing an auxiliary vector field}

  \begin{proposition}\label{vert-vf}
    Let $G$ be a complex Lie group acting holomorphically on a
    complex manifold $X$. Assume that all isotropy groups
    are connected and one-dimensional.

    Let $Z$ be a non-compact Riemann surface.
    
    Let $C,F:Z\to X$ be holomorphic maps and let
    \[
    V=\{(g,t)\in G\times Z: g\cdot C(t)=F(t)\}
    \]

    Let $\pi:V\to Z$ be the natural projection map:
    $\pi(g,t)=t$.

    Then there is an integrable holomorphic vector field
    $\X$ on $V$ such that for every $t\in Z$ the associated flow
    (one-parameter-group) stabilizes $V_t=\pi^{-1}(\{t\})$ and
    coincides with the action of the isotropy group
    \[
    G_{F(t)}=\{g\in G: g\cdot F(t)=F(t)\}
    \]
    on $V_t$ (with $g:(h,t)\mapsto (gh,t)$).
    \end{proposition}

  \begin{proof}
    We define a line bundle on $Z$ by associating to each point
    $p\in Z$ the Lie algebra of
    \[
    \{g\in G:g \cdot C(p)=C(p)\}
    \]
    On a non-compact Riemann surface every holomorphic line bundle is
    trivial.
    Thus we obtain $\X$ as any nowhere vanishing section of this
    trivial holomorphic line bundle.
  \end{proof}

\section{Reduction to the case $\Tr=0$}

\begin{lemma}\label{traceless}
  Let $ \Ac$ be a $\C$-algebra with an antiinvolution $(\ )^ c$
  such that $Z=\{x\in\Ac:x=x^ c\}$ is central in $\Ac$.
  Let $\GC$ be a complex Lie group acting by $\C$-algebra homomorphisms
  on $\Ac$ such that the $Z$ is fixed pointwise.
  Let $\dom$ be a domain in $\C$.
  Let $F,H:\dom\to \Ac$ be holomorphic maps.
  Assume that $\Nm(F)=\Nm(H)$ and $\Tr(F)=\Tr(H)$
  (with $\Tr(F)=F+F^ c$ and $\Nm(F)=FF^ c$).

  Define $\hat F=\frac 12(F-F^c)$ and
  $\hat H=\frac 12(H-H^c)$.

  Then:
  \begin{enumerate}
  \item
    $\Tr(\hat F)=0=\Tr(\hat H)$.
  \item
    $\Nm(\hat F)=\Nm(\hat H)$.
  \item
    There exists a holomorphic map $\phi:\dom\to \GC$
    with $F=\phi(H)$ if and only
    there exists a holomorphic map $\phi:\dom\to \GC$
    with $\hat F=\phi(\hat H)$.
  \end{enumerate}
\end{lemma}

\begin{proof}
  \begin{enumerate}
  \item
    \[
    \Tr(\hat F)=\frac 12\Tr(F-F^ c)=\frac 12\left( F+F^ c-(F^ c+F)\right)=0
    \]
    and similarily $\Tr(\hat H)=0$.
  \item
    Let $F_0=\frac 12\Tr(F)$ and $H_0=\frac 12\Tr(H)$.
    Then $F_0^ c=F_0$ and $H_0^ c=H_0$.
    Observe that $F_0$ and $H_0$ are central in $\Ac$.
    
    By construction
    \[
    F=F_0+\hat F,\ H=H_0+\hat H,\ \hat F^c=-\hat F,
    \hat H^c=-\hat H
    \]
    We obtain:
    \begin{align*}
    \Nm(F)&=\Nm(F_0+\hat F)=\left(F_0+\hat F\right)\cdot
    \left( F_0^ c +(\hat F)^ c\right)\\
&    =\left(F_0+\hat F\right)\cdot
    \left( F_0-\hat F\right)\\
&    =
    F_0^2-F_0\hat F+\hat FF_0-(\hat F)^2\\
&    =
    F_0^2-\hat FF_0+\hat FF_0-(\hat F)^2
    \text{ because $F_0$ is central}\\
&    =
    F_0^2-(\hat F)^2\\
&    = \frac 14\left(\Tr(F)\right)^ 2+\Nm(\hat F)\\
    \end{align*}
    
    Similarily we obtain
    \[
    \Nm(H)=\frac 14\left(\Tr(H)\right)^ 2+\Nm(\hat H)
    \]
    In combination with $\Nm(F)=\Nm(H)$ and $\Tr(F)=\Tr(H)$ this yields
    $\Nm(\hat F)=\Nm(\hat H)$.
  \item
    Note that the image of $\Tr:\Ac\to\Ac$ is contained in the center
    of $\Ac$ which is pointwise stabilized by the $\GC$-action.
    It follows that for every holomorphic map
    $\phi:\dom\to \Ac$ we have
    
    \begin{align*}
      &    F(z)=\phi(z)\left(H(z)\right)\  \forall z\in\dom\\
      &    F_0(z)+\hat F(z) =
     \underbrace{\phi(z)\left(H_0(z)\right)}_{=H_0(z)=F_0(z)}
      +\phi(z)\left(\hat H(z)\right)\\
      &    \hat F(z) =
      \phi(z)\left(\hat H(z)\right)\\
    \end{align*}
      \end{enumerate}
\end{proof}

\section{Vanishing orders}

\begin{proposition}\label{prop-reduce}
Let $X$ be a non-compact Riemann surface, $V$ a vector space
and
let $f,g:X\to V$ be holomorphic maps which do not vanish
identically.

Assume that $f,g$ have the same divisor
(as defined in Definition~\ref{def-vd}).

Then there exist a holomorphic function
$h:X\to\C$ and holomorphic maps $\tilde f, \tilde g: X\to V\setminus\{0\}$
such that $f=h\tilde f$, $g=h\tilde g$.
\end{proposition}

\begin{proof}
  Recall that on a non-compact Riemann surface every divisor
  is a {\em principal} divisor, i.e., the divisor of a holomorphic
  function.

  We choose a holomorphic function $h$ on $X$ with
  \[
  div(h)=div(g)=div(f)
  \]
  and define $\tilde f=f/h$, $\tilde g=g/h$.
\end{proof}

\begin{lemma}\label{div-unchanged}
Let $X$ be a Riemann surface, $V$ a vector space
and
$f:X\to V$, $\phi:X\to GL(V)$ be holomorphic maps.
Assume that $f$ is not vanishing identically.

Define $g(z)=\phi(z)\left(f(z)\right)$.

Then $f$ and $g$ have the same divisor.
\end{lemma}

\begin{proof}
  Let $div(f)=\sum_p m_p\{p\}$. Then for every $p\in X$ and
  $i\in\{1,\ldots,n\}$
  the germ of $f_i$ at $p$ is a divisible by $z_p^{m_p}$ where
  $z_p$ is a local coordinate with $z_p(p)=0$.
  Since $\phi(p)$ is linear, the components $g_i$ likewise have
  germs at $p$ which are 
  divisible by $z_p^{m_p}$.
  Hence $div(g)\ge div(f)$.

  The same arguments show that also $div(f)\ge div(g)$,
  since
  \[
  f(z)=\tilde\phi(z)\left(g(z)\right)
  \]
  for
  \[
  \tilde\phi(z)=\left(\phi(z)\right)^{-1}
  \]

  Thus $div(f)=div(g)$.
\end{proof}

\section{Some cohomology}

We need the lemma below as a preparation for the proof of
Proposition~\ref{coh}.

\begin{lemma}
Let $X$ be a Riemann surface, $D\subset X$ a discrete subset,
$\tau_{0}:P\to X\setminus D$ an unramified $2:1$-covering.

Then $\tau_0$ extends to a ramified covering $\tau:X'\to X$,
i.e., we have a commutative diagram:
\[
\begin{tikzcd}
&  P \arrow[r, hook] \arrow[d,"\tau_0"]& X' \arrow[d,"\tau"]\\
&  X\setminus D \arrow[r,hook]& X
\end{tikzcd}
\]

\end{lemma}

\begin{proof}
  Let $q\in D$ and let $\psi:U\to\Delta$ be a coordinate chart
  with $\psi(q)=0$ and $U\cap D=\{q\}$.
  
  An unramified $2:1$ covering over $\Delta^ *=\Delta\setminus\{0\}$
  is either a product or given by a group homomorphism
  $\rho:\pi_1(\Delta^*)\to\Z/2\Z$.
  There is only one non-trivial group homomorphism from
  $\Z\cong\pi_1(\Delta^*)$ to $\Z/2\Z$.

  Hence either $\tau_0$ restricts on $U^ *=U\setminus\{q\}$
  to a direct product $U^ *\times \Z/2\Z$
  (in which case $\tau_0$ trivially extends through $q$),
  or we can identify the restriction to $U^*$ as the only
  non-trivial $2:1$-covering, which can be realized as
  $z\mapsto z^ 2$ as a map from $\Delta^ *$ to $\Delta^ *$.
  Thus we obtain the following commutative diagram:

  \begin{center}
   \begin{tikzcd}
    \tau^ {-1}(U^ *) \arrow[r,"\sim"] \arrow[d,"\tau_0"] &
    \Delta^ * \arrow[d,"z\mapsto z^ 2"]\\
    U^ * \arrow[r,"\psi"] & \Delta^*
  \end{tikzcd}
  \end{center}
  
  Now the covering map $z\mapsto z^ 2$ obviously extends from
an unramified covering $\Delta^ *\to\Delta^*$
  to a ramified covering $\Delta\to\Delta$.

  Performing this procedure around every point of $D$, we obtain the
  desired extension.
\end{proof}

\begin{proposition}\label{coh}
  Let $X$ be a non-compact Riemann surface, $D$ a discrete subset,
  $\cS$ a sheaf which is locally isomorphic to $\underline{\Z}$
  (the sheaf of locally constant $\Z$-valued functions) on $X\setminus D$.
  Let
  \begin{equation}\label{ex-sec-11}
  0 \to \cS \to \O_X \to \As \to 0
  \end{equation}
  be a short exact sequence of $\O_X$ module sheaves.

  Then $H^1(X,\As)=\{0\}$ and $H^2(X,\cS)=\{0\}$.
\end{proposition}

\begin{proof}
  $X$ is Stein, hence $H^k(X,\O)=0$ for $k>0$. Therefore the
  long exact cohomology sequence associated to the above sequence
  of sheaves \eqref{ex-sec-11} implies
  \[
  H^1(X,\As)\cong H^2(X,\cS).
  \]
  Let $\cS_0$ be the sheaf on $X$ defined by
  \[
  \cS_0(U)=
  \begin{cases} \{0\} & \text{ if $U\cap D\ne\{\}$} \\
    \cS(U) & \text{ if $U\cap D=\{\}$ } \\
  \end{cases}
  \]
  Then we have a short exact sequence
  \begin{equation}\label{ssf}
  0 \to \cS_0 \to \cS \to \cF \to 0
  \end{equation}
  where $\cF$ is a skyscraper sheaf supported on $D$.
  Thus $H^k(X,\cF)=0$ for $k>0$. Consequently
  the long exact cohomology sequence associated to
  \eqref{ssf} implies
  \[
  H^2(X,\cS)\cong H^2(X,\cS_0).
  \]
  By construction, $\cS_0$ is a sheaf on $X\setminus D$
  locally isomorphic to $\underline{\Z}$.
  $(\Z,+)$ admits only one non-trivial automorphism,
  namely $n\mapsto -n$. Hence we may regard $\cS_0$ as the sheaf
  of sections in a locally trivial fiber bundle $B\to X\setminus D$
  with structure
  group $\Z/2\Z$ and typical fiber $\Z$.
  We consider the sub-bundle $P$ with typical fiber $\{+1,-1\}$.
  Such a bundle is 
  an unramified $2\!:\!1$-covering $\tau_0:P\to X\setminus D$.
  It extends to a ramified $2\!:\!1$-covering $\tau:X'\to X$ by adding
  one point above each point of $D$.

  We consider the natural sheaf homomorphism $\uZ_X\to\tau_*\uZ_{X'}$
  given by $f\mapsto f\circ\tau$.
  We define a sheaf homomorphism $\alpha$ from $\tau_*\uZ_{X'}$ to $\cS_0$ as follows:
  Given a $\Z$-valued function $f$ on $\{-1,+1\}$,
  let
  \[d=|f(1)-f(-1)|,\]
  and associate
  \[
  \alpha(f)\stackrel{def}{=}
  \begin{cases} d & \text{ if $f(1)\ge f(-1)$ }\\
    -d & \text{ if $f(1)< f(-1)$ }\\
  \end{cases}
  \]
  This yields a short exact sequence of sheaves
  \[
  0 \to \uZ_X \longrightarrow \tau_*\uZ_{X'}
  \stackrel{\alpha}{\longrightarrow} \cS_0 \to 0
  \]

  We note that outside $D$ the covering map $\tau$ is locally biholomorphic,
  while for a point $p\in D$, the preimage of a small neighborhood of $p$
  in $X$ will be a small neighborhood of $\tau^{-1}(p)$ in $X'$ and
  will be contractible, since $X'$ is a complex space.

  It follows that the higher direct image sheaves $R^q\tau_*\uZ$
  ($q>0$) are trivial. We consider the Leray spectral sequence
  for the sheaf $\uZ$ on $X'$
  and the map $\tau:X'\to X$:
  \[
  H^{p+q}(X',\uZ)=E^{p+q}\,\Leftarrow\, E^{pq}_2=H^p(X,R^ q\tau_*\uZ)
  \]
  Since $R^q\tau_*\uZ=0\,\,\forall q>0$, this spectral sequence
  degenerates.
  Therefore
  \[
  H^k(X,\tau_*\underline{\Z})\cong H^k(X',\uZ)\cong H^k(X',\Z)
  \]
  
  Since $X'$ and $X$ are non-compact Riemann surfaces, we have
  $H^k(X',\Z)=H^k(X,\Z)=0$ for $k\ge 2$ and therefore
  \[
  \ldots \to \underbrace{H^2(X',\Z)}_{=0}
  \to H^2(X,\cS_0) \to \underbrace{H^3(X,\Z)}_{=0}\to\ldots
  \]
  Therefore
  \[
  H^1(X,\As)\cong H^2(X,\cS)\cong H^2(X,\cS_0)=\{0\}.
  \]
\end{proof}

\section{Automorphisms and conjugacy}

For $\H$ and $\Hc$ every $\R$-(resp.~$\C$-) algebra
automorphism $\phi$ is {\em inner}, i.e., given by conjugation
with an element: $\exists \,\, q:\phi:x\mapsto qxq^{-1}$.
Conversely for every $q\in\H$ (resp.~$q\in\Hc$) conjugation by
$q$ defines an automorphism. This automorphism is trivial if and
only if $q$ is in the center.

Therefore we have a short exact sequence of Lie groups.
\begin{equation}\label{short-exact-conju}
1 \to \C^* \to \Hc^* \stackrel{\zeta}{\longrightarrow} \Gc  \to 1
\end{equation}
with $\zeta(g):x\mapsto gxg^{-1}$.

In particular, $\Gc\cong\Hc^ */\C^ *$, which implies that
$\Hc^*$ is a $\C^*$-principal bundle over $\Gc$
(see e.g.~\cite{SM}, Prop.~13.25).

\begin{proposition}\label{p13}
  Let $D$ be an open subset in $\C$ and let $\phi:D\to\Gc$
  be a holomorphic map.

  Then there exists a holomorphic map $\tilde\phi:D\to\Hc^*$ with
  $\phi=\zeta\circ\tilde\phi$.
\end{proposition}

\begin{proof}
  Due to \eqref{short-exact-conju} we have a $\C^*$-principal bundle on
  $\Gc$ which we may pull back
  via $\phi$ to obtain a $\C^*$-principal bundle on $D$.
  But $D$ is a (not necessarily connected) Stein Riemann surface
  and therefore every $\C^*$-principal bundle on $D$ is holomorphically
  trivial.
  Thus it admits a section. Such a section corresponds to a
  lifting as below
  \[
  \begin{tikzcd}
    & \Hc^* \arrow[d,"\zeta"] \\
    D \arrow[r,"\phi"] \arrow[ur,"\tilde\phi"] & \Gc \\
  \end{tikzcd}
  \]
\end{proof}
  \section{Proof of the Main Theorem}\label{pf-main}

We are now in a position to prove our Main Theorem
\ref{mainth}.

First we prove the most difficult part of Theorem~\ref{mainth}
which is the following.

\begin{theorem}\label{mainmainth}
  Let $A=\H$ and $A_\C=A\tensor_\R\C$.

  Let $\dom\subset\C$ be a domain which is invariant under
  conjugation.
  
  Let $F,H:\dom\to A_\C\setminus\{0\}$ be holomorphic maps
  such that $\overline{F(\bar z)}=F(z)$,
  $\overline{H(\bar z)}=H(z)$,
  $\Tr(F)=\Tr(H)=0$, $\Nm(F)=\Nm(H)$.%
\footnote{We need no condition on $\cdiv$, because the
  assumptions  $\Tr(F)=\Tr(H)=0$ and 
  $F(\dom),H(\dom)\subset A_\C\setminus\{0\}$
  {\em imply} that $\cdiv(F)$ and $\cdiv(H)$ are empty.}

  Then there exists a holomorphic map $\phi:\dom\to\GC$ such that
  \[
  \phi(z)\left(F(z)\right)=H(z)\ \forall z\in\dom.
  \]
\end{theorem}

\begin{proof}
  Throughout the proof we will use
    the fact that $\dom$ is a non-compact Riemann surface.

    Note that we assume $\Tr(F)=\Tr(H)=0$.
    It follows that the images $F(\dom),H(\dom)$ are contained in $\WC$.

    Thus we may regard $F$ and $H$ as holomorphic maps from $\dom$ to
    $\left(\WC\right)\setminus\{0\}=
    \left(\WC\right)\cap\left(A_\C\setminus\{0\}\right)$.

From our assumption on $F$ and $H$ we deduce that for every $t\in\dom$ there
is an element $g\in G_\C$ with $H(t)=g(F(t))$ (Corollary~\ref{aut-if-n-tr}).

The case where both $F$ and $H$ are constant is trivial.
Hence we may assume that at least one of the two maps is not constant.
Wlog $F$ is not constant.

We define
\[
V=\{(\alpha,z)\in\GC\times\dom : F(z)= \alpha H(z)\}
\]

The isotropy groups for the
$\GC$-action on $(\WC)\setminus\{0\}$ have all the same
dimension (namely $1$ if $A=\H$)
due to Corollary~\ref{dim-c-isotropy}. 

Therefore we may apply Proposition~\ref{loc-eq}.
It follows that for every $p\in\dom$ there is an open neighborhood $U$
of $p$ in $\dom$
and a holomorphic map $\psi:U\to \GC$ with
\[
F(z)=\psi(z)H(z)\ \,\, \forall z\in U.
\]

Thanks to Proposition~\ref{vert-vf} we obtain a vector field which
gives us a one-parameter group acting transitively on the fibers
of $V\to\dom$, i.e., there is an action
\[
\mu:(\C,+)\times V\to V.
\]

We define  $\As$ as the sheaf on $\dom$
of holomorphic maps $\xi$ from $\dom$ to $G_\C$ such that
$\xi(t)$ is contained in the isotropy group
\[
G_{F(t)} =\{g\in G: g(F(t))=F(t)\}
\]
for every
$t$.

Due to Proposition~\ref{vert-vf} we may identify $\As$
with a sheaf 
of fiber-preserving automorphisms, namely for an open subset $U\subset\dom$
we define $\As(U)$ as the set of biholomorphic self maps of 
$\pi^{-1}(U)$
which may be written 
as
\[
\eta_\zeta : x\mapsto \mu(\zeta(\pi(x)),x)
\]
for some holomorphic function $\zeta\in\O(U)$.

The map associating $\eta_\zeta$ to $\zeta$ defines a morphism of
sheaves on $\dom$. Let $\cS$ denote its kernel. 
Then we have a short exact sequence
of sheaves
\[
0 \to \cS \longrightarrow
\O \stackrel{\zeta\mapsto \eta_\zeta}%
   {\xrightarrow{\hspace*{1cm}}}
     \As\to 0
\]
$\cS$ may be regarded as the sheaf of those holomorphic functions $\zeta$
for which $\mu(\zeta(t))$ equals the identity map on the fiber
$\pi^{-1}(t)$.

Generically, the isotropy group of $\GC$ at a point in $\Ac$ is
isomorphic to $\C^*$. Elsewhere, the isotropy group is $\C$.
Hence the sheaf $\cS$ is locally isomorphic to
$\underline{\Z}$ on an open subset $\Omega$
of $\dom$ and trivial on the complement $\dom\setminus\Omega$.

Next we want to show that $H^ 1(\dom,\As)=\{0\}$.

If $\Omega$ is empty, then $\cS$ is trivial and
consequently $\As\cong¸\mathcal O$.
This implies $H^ 1(\dom,\As)=\{0\}$,
because $\dom$ is Stein.

Thus we may assume that $\Omega$ is not empty.
  Then $\Omega $ is Zariski open and dense and its complement
  $\dom\setminus\Omega$ is discrete.
Then $H^1(\dom,\As)=0$ follows from
Proposition~\ref{coh}.

We choose an open cover $(U_i)_i$ of $\dom$ with holomorphic
maps $\psi_i:U_i\to\GC$ such that
\[
F(z)=\psi_i(z)H(z), \,\,\forall z\in U_i
\]
(This is possible thanks to Proposition~\ref{loc-eq}.)

On each intersection $U_{ij}=U_i\cap U_j$ we have
\[
F(z)=\psi_i(z)H(z)\text{ and }
F(z)=\psi_j(z)H(z),
\]
implying that
\[
\psi_{ji}(z)\stackrel{def}{=}\psi_j(z)\circ\left(\psi_i(z)\right)^{-1}
\]
is contained in the isotropy
group at $F(z)$.

Thus $\psi_{ij}$ defines a $1$-cocyle of $\As$. Since
$H^1(\dom,\As)=0$, this cocyle is a coboundary, i.e.,
there are $\phi_i\in\Gamma(U_i,\As)$
with $\psi_{ij}=\phi_i(\phi_j)^{-1}$.

Therefore
\begin{align*}
  &\psi_i(z)(\psi_j(z))^{-1} 
  &=\ &\psi_{ij}(z) =\phi_i(z)(\phi_j(z))^{-1} \\
  \implies\quad & \psi_i(z) &=\ & \phi_i(z)(\phi_j(z))^{-1}(\psi_j(z)) \\
  \implies\quad & ( \phi_i(z))^{-1}\psi_i(z) &=\ &(\phi_j(z))^{-1}(\psi_j(z)) \\
\end{align*}

  It follows that
\[
\phi(z)=( \phi_i(z))^{-1}\psi_i(z)
\]
is well-defined.
Since $F(z)=\psi_i(z) H(z)$ and $\phi_i(z)$ is in the isotropy group
of $G_{\mathbb{C}}$ at $F(z)$, we infer
\[
F(z)=\phi(z) H(z)
\]
This completes the proof.

\end{proof}

\begin{proof}[Proof of the Theorem~\ref{mainth}]

  First we deal with the case where neither $f$ nor $h$
    is slice preserving.
  
  We proceed as follows:
\[
  \begin{tikzcd}
    (i) \arrow[r,Leftrightarrow]
    & (ii) \arrow[r,Leftrightarrow]
     \arrow[d,Rightarrow]
    & (iii)\\
    (v) \arrow[r,Leftrightarrow] & (iv)
    \arrow[ur,Rightarrow]
      \\
  \end{tikzcd}
\]

  $(ii)\implies(iv)$:

  By assumption we have  $\Tr(F)=\Tr(H)$.
  Define
  \[
  \hat F=\frac 12\left(F-F^ c\right),\quad
  \hat H=\frac 12\left(H-H^ c\right)
  \]
  Evidently $\Tr(\hat H)=\Tr(\hat F)=0$.
  Moreover $\Nm(F)=\Nm(H)$ in combination with Lemma~\ref{traceless}
  implies that $\Nm(\hat F)=\Nm(\hat H)$.

  With respect to the decomposition $A_\C=Z\oplus\WC$ the map
  $\hat F$ resp.~$\hat H$ is just the second component of
  $F$ resp.~$H$.
  By the definition of the central divisor (introduced in
  \paragraph\ref{ss-cdiv}) this implies that
  $\cdiv(F)=\cdiv(\hat F)$ and $\cdiv(H)=\cdiv(\hat H)$.
  
  Since $\cdiv(F)=\cdiv(H)$, it follows that there are holomorphic maps
  $\tilde F,\tilde H:\dom\to\Ac\setminus\{0\}$ and
  $\lambda:\dom\to\C$ such that
  \[
  \hat F=\lambda\tilde F, \quad\hat H=\lambda\tilde H
  \]
  (Proposition~\ref{prop-reduce}).
  Observe that
  \begin{align*}
  &0=\Tr(\hat F)=\lambda Tr(\tilde F),\ \Nm(\hat F)=\lambda^ 2N(\tilde F)\\
  &0=\Tr(\hat H)=\lambda Tr(\tilde H),\ \Nm(\hat H)=\lambda^ 2N(\tilde H)\\
  \end{align*}
  and that $\lambda$ does not vanish identically,
    because $f$ and $h$
    are not slice preserving.
  Hence $\Tr(\tilde F)=0=\Tr(\tilde H)$ and $\Nm(\tilde F)=\Nm(\tilde H)$
  and Theorem~\ref{mainmainth} implies that there is a holomorphic map
  $\phi:\dom\to\GC$ such that
  \[
  \phi(z)\left(\tilde F(z)\right)=\tilde H(z)\ \forall z\in\dom.
  \]
  which in turn implies
  \[
  \phi(z)\left(\hat F(z)\right)=\hat H(z)\ \forall z\in\dom,
  \]
  because $\Gc$ acts linearly,
  $\hat F=h\tilde F$ and $\hat H=h\tilde H$.
  Finally
  \[
  \phi(z)\left(F(z)\right)=H(z)\ \forall z\in\dom
  \]
  follows via Lemma~\ref{traceless}.
  
  $(iv)\implies(iii)$:
  The implication $(iv)\implies \cdiv(F)=\cdiv(H)$ is due to
  Lemma~\ref{div-unchanged}, the other assertion is obvious.
  
  For $(iii)\iff(ii)$ see Corollary~\ref{aut-if-n-tr}.

  For $(i)\iff(ii)$, see
  Proposition~\ref{n-tr-stem} and
  Section \ref{ss-cdiv}. 

    $(v)\implies(iv)$. Every $\alpha\in\Hc^*$ defines an automorphism
  of $\Hc$ via $q\mapsto \alpha q\alpha^{-1}$. This defines a natural
  map from $\Hc^*$ to $\Gc$. Composition with this map yields the
  desired implication.

  $(iv)\implies(v)$ follows from Proposition~\ref{p13}.

    This finishes the proof for case $a)$.
  Let us deal now with the case $b)$, i.e. we assume that
  $f$ is slice preserving.

  $(i)\iff(ii)$ is due to the correspondence between slice regular functions
  and stem functions.

  $(ii)\implies(v)\implies(iv)\implies(iii)$ is trivial.

  $(iii)\implies(ii)$: Because $f$ is slice preserving, all the values
  of $F$ are contained in the center $Z\cong\C$ of $\Hc$. But $\GC$
  fixes the center pointwise. Hence the statement.
     \end{proof}

\nocite{LF,FH,SV,F2,B,G,MM,F,BW2,BW1,BW3,AB2,AB1,BS,BG,BM,GSS,GS2,GS1,GP,GPS}

\bibliography{auto}
\bibliographystyle{alpha}
\end{document}